\documentclass[12pt]{amsart}
\usepackage{a4wide,enumerate,xcolor,amssymb}
\usepackage{amsmath,comment,graphicx}
\allowdisplaybreaks
\usepackage{enumitem}
\usepackage{float}
\setlist[enumerate]{label={\rm(\arabic*)}, leftmargin=28pt, itemsep=3pt}
\setlist[itemize]{label={$\bullet$}, leftmargin=12pt, itemsep=3pt}

\let\pa\partial
\let\na\nabla
\let\eps\varepsilon
\newcommand{\N}{{\mathbb N}}
\newcommand{\R}{{\mathbb R}}
\newcommand{\diver}{\operatorname{div}}

\DeclareMathOperator*{\argmin}{arg\,min}

\newcommand{\graysquare}{\textcolor{gray}{\blacksquare}}
\newcommand{\qedd}{{\hfill$\graysquare$}}

\newtheorem{theorem}{Theorem}
\newtheorem{lemma}[theorem]{Lemma}
\newtheorem{proposition}[theorem]{Proposition}
\newtheorem{remark}[theorem]{Remark}

\newtheorem{definition}{Definition}

\begin{document}

\title[Cross-diffusion systems seen as Wasserstein gradient flows]{A review of compactness methods for cross-diffusion systems seen as Wasserstein gradient flows}

\author[M. Dus]{Mathias Dus}
\address{Institute of Analysis and Scientific Computing, TU Wien, Wiedner Hauptstra\ss e 8--10, 1040 Wien, Austria \& \newline IRMA, UMR 7501, Universit\'e de Strasbourg, 7 rue Ren\'e-Descartes, 67084 Strasbourg Cedex, France}
\email{dus@unistra.fr}

\author[A. J\"ungel]{Ansgar J\"ungel}
\address{Institute of Analysis and Scientific Computing, TU Wien, Wiedner Hauptstra\ss e 8--10, 1040 Wien, Austria}
\email{juengel@tuwien.ac.at} 

\date{\today}

\thanks{The authors acknowledge partial support from   
the Austrian Science Fund (FWF), grant 10.55776/PAT2687825, and from the Austrian Federal Ministry for Women, Science and Research and implemented by \"OAD, project MultHeFlo. This work has received funding from the European Research Council (ERC) under the European Union's Horizon 2020 research and innovation programme, ERC Advanced Grant NEUROMORPH, no.~101018153. For open access purposes, the authors have applied a CC BY public copyright license to any author-accepted manuscript version arising from this submission.} 

\begin{abstract}
A comprehensive methodology for establishing the existence of gradient flows for cross-diffusion systems with respect to suitable energies is proposed. The approach is based on the construction of piecewise-in-time constant approximations via the Jordan--Kinderlehrer--Otto scheme. Compactness of the approximate sequence is obtained using either the flow interchange technique or the five gradient inequality. These methods are illustrated for both parabolic and hyperbolic--parabolic Busenberg--Travis systems, as well as for several of their variants. This paper reviews the results from the literature and discusses additional properties. 
\end{abstract}

% \paragraph{Keywords:}  
 
% \paragraph{AMS classification:}  

\maketitle

%%%%%%%%%%%%%%%%%%%%%%%%%%%%%%%%%%%%%%%%%%%%%%%%%%%%%%%%%%%%%%%%%%%

\section{Introduction}

Many applications consist of multiple components whose diffusive transport can be modeled by parabolic cross-diffusion systems. The state is given by the densities $u_1,\ldots,u_N$ of $N$ species interacting with each other. Examples include gas mixtures, competing population species, enzymes in biological cells, and neural networks. Neglecting drift terms and reactions and focusing on diffusion phenomena, it is often possible to write the cross-diffusion systems in the form
\begin{align}\label{1.eq}
  \pa_t u_i = \diver\bigg(u_i\na\frac{\pa\mathcal{E}}{\pa u_i}(u)\bigg)
  \quad\mbox{in }\Omega, \quad 
  u_i\na\frac{\pa\mathcal{E}}{\pa u_i}(u)\cdot\nu = 0\quad\mbox{on }\pa\Omega,
  \ t>0,
\end{align}
where $\Omega\subset\R^d$ ($d\ge 1$) is a convex compact domain, $\nu$ denotes the exterior normal unit vector to the boundary $\pa\Omega$, $\mathcal{E}$ is the energy of the system, and we impose the initial conditions $u_i(0)=u_i^0$ in $\Omega$, $i=1,\ldots,N$. 

System \eqref{1.eq} admits a Wasserstein gradient flow structure, where the evolution can be seen as a steepest descent of the energy functional in a metric space of measures, providing a clear geometric interpretation of the flow and a connection to thermodynamic principles. Moreover, the structure enables a powerful existence theory via minimizing movements as well as stability and uniqueness results through metric convexity. Surprisingly, the question of whether system \eqref{1.eq} can be formulated rigorously as a Wasserstein gradient flow is very delicate, and there are only few results in the literature. 

In this paper, we review some compactness methods for cross-diffusion systems that can be formulated as a Wasserstein gradient flow, focusing on the generalized Busenberg--Travis equations (and their variants)
\begin{equation}\label{1.BT}
\begin{aligned}
  & \pa_t u_i = \diver(u_i\na p_i(u)), \quad 
  p_i(u) = \sum_{j=1}^N a_{ij}u_j\quad\mbox{in }\Omega,\ t>0, \\
  & u_i\na p_i(u)\cdot\nu = 0\quad\mbox{on }\pa\Omega,\quad
  u_i(0)=u_i^0\quad\mbox{in }\Omega,\ i=1,\ldots,N.
\end{aligned}
\end{equation}
If $A=(a_{ij})$ is positive definite, equations \eqref{1.BT} are parabolic in the sense of Petrovskii, and we call them the parabolic Busenberg--Travis system. When $A$ has not full rank, system \eqref{1.BT} has a symmetric hyperbolic part \cite[Theorem 2.5]{DHJ23}. In this case, \eqref{1.BT} is denoted as the hyperbolic--parabolic Busenberg--Travis equations.

In the case of two species $N=2$ and $a_{ij}=k_i>0$, equations \eqref{1.BT} have been first suggested in \cite{BuTr83,GuPi84} to describe segregating populations. Equations \eqref{1.BT} with $N=2$ have been derived from an interacting particle system in the many-particle limit \cite{GaSe14}. A derivation for multiple species, assuming moderate interactions and using a mean-field approach, can be found in \cite{CDJ19}. An existence proof for general coefficients $a_{ij}$ was given in \cite{GaSe14} for two species and in \cite[Appendix B]{JPZ22} for an arbitrary number of species, assuming that $(a_{ij})$ is positive definite. 

The case that the matrix $(a_{ij})$ has not full rank was first studied in \cite{BGHP85} for two species and in one space dimensions. Later,  generalizations to multiple species and several dimensions have been investigated, mostly for non-segregating solutions \cite{BHIM12,DHJ23,DrJu20,GSV15,HoJu25,GPS19}. The authors of \cite{CFSS18,KiMe18} were able to prove the existence of segregating weak solutions in one space dimension, while the existence of segregating Lagrangian solutions was shown in \cite{Jac25}. 

Equations \eqref{1.BT} can be rigorously formulated as a Wasserstein gradient flow in the space of probability measures. This was shown for positive definite matrices $(a_{ij})$ and in one space dimension in \cite{LaMa13}. A weak solution to \eqref{1.BT} was shown to exist as the limit of the minimizers of the Jordan--Kinderlehrer--Otto (JKO) scheme. This approach was also applied to the hyperbolic--parabolic model; see, e.g., \cite{CFSS18,Lab20}. Related Busenberg--Travis models include, for instance, dispersal potentials \cite{Igb24} or nonlocal approximations \cite{DiFa13,MaPa24}. 

We also mention other cross-diffusion systems that can be formulated as a Wasserstein gradient flow, like Keller--Segel models \cite{BCC08}, Cahn--Hilliard-type fourth-order systems \cite{MaZi17}, and incompressible immiscible multiphase flows \cite{CGM17}, but we focus on the Busenberg--Travis model.

In this paper, we review compactness methods used to analyze the parabolic Busen\-berg--Travis equations \eqref{1.BT} seen as a Wasserstein gradient flow (Section \ref{sec.ex}). Furthermore, we discuss a fourth-order variant (Section \ref{sec.fourth}) and introduce a new definition of solution to the hyperbolic--parabolic Busenberg--Travis model (Section \ref{sec.hyper}). Finally, we discuss the effects from the additional diffusion term coming from the Shigesada--Kawasaki--Teramoto model, which is related to \eqref{1.BT} (Section \ref{sec.SKT}). For the convenience of the reader, we recall basic elements of optimal transport theory in Section \ref{sec.basics}.

%%%%%%%%%%%%%%%%%%%%%%%%%%%%%%%

\section{Basic elements of optimal transport theory}\label{sec.basics}

We briefly recall the basic notions of optimal transport theory. For a detailed exposition, we refer to the monographs \cite{AGS05,San15}. Let $\Omega$ be a convex compact domain and $P_2(\Omega)$ be the space of probability measures on $\Omega$ with finite second-order moment (which is equal here to the entire space of probability measures since $\Omega$ is compact). Let $\Pi(\mu,\nu)$ for $\mu$, $\nu\in P_2(\Omega)$ be the set of probability measures on $\Omega\times\Omega$ having $\mu$ and $\nu$ as first and second marginals, respectively, i.e., $\gamma\in\Pi(\mu,\nu)$ if for all $\phi$, $\psi\in C_c(\Omega)$,
\begin{align*}%\label{2.marg}
  \int_{\Omega\times\Omega}\phi(x)d\gamma(x,y) 
  = \int_\Omega\phi(x)d\mu(x), \quad
  \int_{\Omega\times\Omega}\psi(y)d\gamma(x,y) 
  = \int_\Omega\psi(y)d\nu(y).
\end{align*}
Let $X$ and $Y$ be two measurable spaces, $\theta:X\to Y$ be a measurable function, and $\mu\in P(X)$. The push-forward $\theta_{\#}\mu$ of $\mu$ is defined by
\begin{align*}
  \theta_{\#}\mu(A) = \mu(\theta^{-1}(A))
  \quad\mbox{for all measurable sets }A \mbox{ of }Y
\end{align*}
or, more generally, for measurable functions $\phi:Y\to\R$,
\begin{align*}
  \int_Y\phi(y)d\theta_{\#}\mu(y) = \int_X\phi(\theta(x))d\mu(x).
\end{align*}
If $\mu=udx$ is absolutely continuous (with respect to the Lebesgue measure) and $\theta$ is a one-to-one Lipschitz continuous function satisfying $\det\theta'\neq 0$ on $\operatorname{supp}(u)$ (with $\theta'$ being the Jacobian of $\theta$), then $\theta_{\#}\mu=:vdx$ is also absolutely continuous and it holds that
\begin{align}\label{2.cov}
  v = \frac{u\circ\theta^{-1}}{\det(\theta'\circ\theta^{-1})}.
\end{align}

For given $\mu$, $\nu\in P_2(\Omega)$, we denote by $W_2(\mu,\nu)$ the {\em Wasserstein distance} between $\mu$ and $\nu$, defined by
\begin{align*}%\label{1.min}
  W_2^2(\mu,\nu) = \inf_{\gamma\in\Pi(\mu,\nu)}
  \int_{\Omega\times\Omega}|x-y|^2 d\gamma(x,y).
\end{align*}
This minimization problem (called Kantorovich's problem) has an optimal solution, which is generally not unique \cite[Theorem 6.1.4]{AGS05}. If the measure $\mu$ is absolutely continuous, the optimal coupling (or optimal transport plan) $\gamma$ is unique and has the specific form $\gamma = (I\times\theta)_{\#}\mu$, where $I$ is the identity, $\theta = I-\na\phi$ is the transport map pushing $\mu$ to $\nu$, and $\phi$ is the so-called {\em Kantorovich potential} \cite[Theorem 6.2.4]{AGS05}. The function $\phi$ is $\lambda$-convex with value $\lambda=1$, i.e., for all $x_0$, $x_1\in\Omega$, and $t\in[0,1]$,
\begin{align*}
  \phi((1-t)x_0+tx_1) \le (1-t)\phi(x_0) + t\phi(x_1)
  + \frac{\lambda}{2}t(1-t)|x_1-x_0|^2.
\end{align*}
It follows from the definition of the Kantorovich potential that
\begin{align}\label{2.W2Kant}
  W_2^2(\mu,\nu) = \int_{\Omega\times\Omega}|x-(x-\na\phi(x))|^2
  d\mu(x) = \int_{\Omega\times\Omega}|\na\phi(x)|^2\mu(x)dx.
\end{align}

We describe the state of a cross-diffusion system by a vector of probability measures for which the corresponding measure is the product metric on $P_2(\Omega)^N$, defined by
\begin{align*}
  W_2^2(\mu,\nu) := \sum_{i=1}^n W_2^2(\mu_i,\nu_i)
\end{align*}
for all $\mu=(\mu_1,\ldots,\mu_n)$, $\nu=(\nu_1,\ldots,\nu_n)\in P_2(\Omega)^N$.

The metric space $(P_2(\Omega),W_2)$ is a length space in the sense of the following definition.

\begin{definition}
Let $(X,d)$ be a metric space. For all $(x,y) \in X^2$, define the distance
\begin{equation}\label{eq:def_length_space}
  d_\ell(x,y) := \inf_{\stackrel{\eta \in C([0,1]; X),}{\eta(0)=x,\,
  \eta(1)=y}}\sup_{\stackrel{0 = t_0 < t_1 < \cdots < t_N = 1,}{N\in\N}}
  \sum_{i=0}^{N-1} d(\eta(t_i), \eta(t_{i+1})).
\end{equation}
If $d=d_\ell$, then the metric space $(X,d)$ is said to be a length space. If the infimum in \eqref{eq:def_length_space} is reached for some curve $\eta$, we say that $\eta$ is a geodesic linking $x$ to $y$. It is an easy but tedious exercise to prove that one can reparametrize $\eta$ so that it is a constant speed geodesic, i.e.
\begin{equation}\label{eq:constant_speed_geodesic}
  d(\eta_s,\eta_t) = (t-s) d(x,y)
  \quad\mbox{for all }0 \leq s \leq t \leq 1.
\end{equation}
The metric space $(X,d)$ is said to be a geodesic space if such an infimum is reached for all couples $(x,y) \in X^2$.
\end{definition} 

\begin{remark}\rm 
Conversely, if a curve $\eta: [0,1] \rightarrow X$ satisfies \eqref{eq:constant_speed_geodesic}, then the infimum in \eqref{eq:def_length_space} is reached and $\eta$ is a geodesic (of constant speed). Indeed, 
$$
  \sum_{i=0}^{N-1} d(\eta(t_i), 
  \eta(t_{i+1})) = \sum_{i=0}^{N-1} (t_{i+1} -t_i)d(x, y) = d(x,y)
$$
for every partition $(t_i)_i$ of $[0,1]$. 
\qedd\end{remark}

The space $(P_2(\Omega),W_2)$ is in fact a geodesic space, due to the following result that is proved in \cite[Theorem 7.2.2]{AGS05}.

\begin{theorem}
For any couple $(\mu,\nu)\in P_2(\Omega)^2$, there exists a (constant speed) geodesic $(\mu_t)_{0\le t\le 1}\in C^0([0,1];P_2(\Omega))$ connecting $\mu$ and $\nu$. Taking an optimal coupling $\gamma\in\Pi(\mu,\nu)$, a constant speed geodesic is constructed as
\begin{align*}
  \mu_t = ((1-t)\pi_1 + t\pi_2)_{\#}\gamma, \quad t\in[0,1],
\end{align*}
where $\pi_1$ and $\pi_2$ are the projection operators given by $\pi_1(x,y)=x$, $\pi_2(x,y)=y$.
\end{theorem}

Observe that a geodesic in $P_2(\Omega)^N$ is the usual product of geodesics in $P_2(\Omega)$. 

It is possible to view $(P_2(\Omega),W_2)$ as a kind of Riemannian manifold with the tangent space \cite[Def.~3.31]{AmGi12}
\begin{align*}
  T_\mu P_2(\Omega) = \overline{\{\na\phi:\phi\in 
  C_0^\infty(\Omega)\}}^{L^2(\mu)}.
\end{align*}
To understand this, we need the notion of absolutely continuous curve. A curve $\eta:[0,T]\to X$ of a metric space $(X,d)$ belongs to the class of {\em absolutely continuous curves} $AC^2([0,T];X)$ if there exists $f\in L^2(0,T;\R^+)$ such that
\begin{align}\label{2.f}
  d(\eta(s),\eta(t)) \le \int_s^t f(\tau)d\tau
  \quad\mbox{for all }0\le s\le t\le T.
\end{align}
For any $\eta\in AC^2([0,T];X)$, the limit
\begin{align*}
  |\eta'(t)| := \lim_{s\to t}\frac{d(\eta(s),\eta(t))}{|s-t|} 
\end{align*}
exists almost everywhere and is minimal in the sense $|\eta'(t)|\le f(t)$ for almost all $t\in[0,T]$ and for any $f$ satisfying \eqref{2.f} \cite[Theorem 1.1.2]{AGS05}. The function $|\eta'|$ is called the {\em metric derivative} of $\eta$.

Absolutely continuous curves satisfy a continuity equation. This is formalized in the following fundamental theorem \cite[Theorem 8.3.1]{AGS05}.

\begin{theorem}\label{thm.fund}
Let $(\mu_t)$ be an absolutely continuous curve in $(P_2(\Omega),W_2)$. Then there exists a unique Borel vector field $(t,x)\mapsto v_t(x):=v(t,x)$ such that $v_t\in T_{\mu_t}P_2(\Omega)$ for a.e.\ $t\in[0,T]$ and
\begin{align*}
  \pa_t\mu_t + \diver(\mu_t v_t) = 0\quad\mbox{in the weak sense},
  \quad \|v_t\|_{L^2(\mu_t)} = |\mu_t'|\quad\mbox{for a.e. } t\in[0,T].
\end{align*}
\end{theorem}

This theorem shows that if we want to move from a measure $\mu_t\in P_2(\Omega)$ to another measure in an absolute continuous way, we need to choose a vector field $v_t$ belonging to $T_{\mu_t}P_2(\Omega)$. For this reason, this space is called a tangent space in analogy with finite-dimensional manifold theory. 

We introduce the $L^2(\mu_t)$ inner product 
\begin{align*}
  g_\mu(v,w) := \int_\Omega v w d\mu
  \quad\mbox{for }v,\, w\in T_\mu P_2(\Omega).
\end{align*}
This definition is consistent with the Wasserstein metric. Indeed, let $(\mu_t)_{0\le t\le 1}$ be a constant speed geodesic and $(v_t)$ be the associated vector field from Theorem \ref{thm.fund}. We infer from \eqref{eq:constant_speed_geodesic} that $W_2(\mu_t,\mu_s)=(t-s)W_2(\mu_1,\mu_0)$ for $s\le t$ and hence
\begin{align*}
  |\mu'_t| = \lim_{s\to t}\frac{W_2(\mu_t,\mu_s)}{|t-s|}
  = W_2(\mu_1,\mu_0),
\end{align*}
showing for $v_t\in T_{\mu_t}P_2(\Omega)$ that
\begin{align*}
  \int_0^1 g_{\mu_t}(v_t,v_t)dt = \int_0^1\|v_t\|_{L^2(\mu_t)}^2 dt
  = \int_0^1|\mu'_t|^2 dt = W_2^2(\mu_1,\mu_0),
\end{align*}
where we used Theorem \ref{thm.fund} to relate the $L^2(\mu_t)$ norm of $v_t$ with the metric derivative of $\mu_t$. 

We are now in the position to make precise the notion of Wasserstein gradient flow for \eqref{1.eq}. Let an energy of the form
\begin{align*}
  \mathcal{E}(u) = \int_\Omega f(x,u)dx
\end{align*}
be given, where $u=(u_1,\ldots,u_N)$ is a vector-valued measure satisfying $u_i dx=du_i$ (we identify the measure and the measurable function). We define the inner product on the product tangent space by
\begin{align*}
  G_u(v,w) = \sum_{i=1}^N g_{u_i}(v_i,w_i)
  \quad\mbox{for }v=(v_i),\, w=(w_i)\in\prod_{i=1}^NT_{u_i}P_2(\Omega).
\end{align*}
By Theorem \ref{thm.fund}, its variation through a geodesic becomes
\begin{align*}
  d\mathcal{E}(u)(v) 
  &= \sum_{i=1}^N \int_\Omega \frac{\partial f }{\partial u_i}(x,u) 
  \partial_t u_i dx
  = \sum_{i=1}^N\int_\Omega \nabla 
  \frac{\partial f}{\partial u_i}(x,u)\cdot  v_i du_i.
\end{align*}
Defining the geometric gradient $\na\mathcal{E}$ by $d\mathcal{E}(u)(v)=G_u(v,\na\mathcal{E})$, we can identify $\na\mathcal{E}$ with $\na(\pa f/\pa u)$. These calculations are formal, since $\pa_t u_i + \diver(u_iv_i) = 0$ holds in the weak sense only. For a rigorous definition of a gradient flow in metric spaces, we refer to \cite[Chapter 1]{AGS05}, where the notion of an upper gradient is developed (to substitute gradient norms). This approach lacks a geometric interpretation, which is the reason why Otto's formalism is often preferred to understand the gradient flow structure. In fact, Otto's calculus is based on a different metric \cite[(18.1)]{ABS21}, and the gradient tangent vectors $v_i=\na\phi_i$ are coupled to the tangent vectors $s_i$ via $\diver(u_i\na\phi_i) = s_i$, which allows us to recover formulation (1), where $\pa\mathcal{E}/\pa u_i$ is identified with $\pa f/\pa u_i$; see, e.g., \cite[Lec.~18]{ABS21} for details.

For later use, we introduce the notion of $\lambda$-geodesically convexity: A functional $\mathcal{E}:D(\mathcal{E})\to\R$ on the domain $D(\mathcal{E})\subset P_2(\Omega)$ is {\em $\lambda$-geodesically convex} for $\lambda\in\R$ if it holds for all geodesics $(u_t)_{0\le t\le 1}$ that
\begin{align*}
  \mathcal{E}(u_t)\le (1-t)\mathcal{E}(u_0) + t\mathcal{E}(u_1)
  - \frac{\lambda}{2}t(1-t)W_2^2(u_0,u_1).
\end{align*}
The functional $\mathcal{E}$ is said to be {\em (geodesically) semi-convex} if there exists $\lambda\in\R$ such that it is $\lambda$-geodesically convex.

%%%%%%%%%%%%%%%%%%%%%%%%%%%%%%%%%

\section{Parabolic Busenberg--Travis model}

The Busenberg--Travis system \eqref{1.BT} possesses two Lyapunov functionals along the solution $u$ to \eqref{1.BT},
\begin{align}\label{3.EH}
  \mathcal{E}(u) = \frac12\sum_{i,j=1}^N\int_\Omega a_{ij}u_iu_j dx,
  \quad \mathcal{H}(u) = \sum_{i=1}^N\int_\Omega u_i(\log u_i-1)dx.
\end{align}
The functional $\mathcal{E}$ can be interpreted as the potential energy and $\mathcal{H}$ as the Boltzmann--Shannon entropy of an underlying fluiddynamical approximation \cite{CCDJ25}. Inserting the variational derivative $\delta\mathcal{E}/\delta u_i = \sum_{j=1}^N a_{ij}u_j$ into \eqref{1.eq} leads to the Busenberg--Travis equations \eqref{1.BT}, while the gradient flow associated to $\mathcal{H}$ corresponds to the heat equation. 

In this section, we show that the energy $\mathcal{E}$ is not semi-convex in the geodesic sense, prove the existence of weak solutions to the parabolic Busenberg--Travis system \eqref{1.BT}, using bounds from the entropy, and discuss a fourth-order variant of \eqref{1.BT}. 

\subsection{Convexity of energies}
%\label{sec.convex}

We show that the energy $\mathcal{E}$ in \eqref{3.EH} is not (geodesically) semi-convex in $L^2(\Omega)^N$, such that we cannot apply the method developed in \cite{AGS05} to prove that the chain rule is satisfied by the associated gradient flow. The lack of semi-convexity is not specific to our choice of energy; see \cite[Prop.~2.4]{BMZ23} for a more general result. 

\begin{lemma}
If $a_{ij}\neq 0$ for some $i\neq j$, then the energy $\mathcal{E}$,  defined in \eqref{3.EH}, is not semi-convex on $P_2(\Omega)^N$.
\end{lemma}

\begin{proof}
Since our setting is slightly different from that one in \cite[Prop.~2.4]{BMZ23}, we give an independent proof. Without loss of generality, we suppose that $i=1$ and $j=2$. Let $n\in\N$, $x_k\in\operatorname{int}(\Omega)$ for $k=1,\ldots,N$ with $x_1\neq x_2$. Set $r=|x_1-x_2|$. We define the measures $u_k^n$ for $k=1,2$ such that $u_k^n$ is constant on the ball $B_{r/2^n}(x_k)$ around $x_k$ with radius $r/2^n$ and $u_k^n$ for $k=3,\ldots,N$ is constant on $B_s(x_k)$ for $0<s<r$, where $x_k$ is chosen outside the cylinder of radius $r$ with axis $\overline{x_1x_2}$. Such a construction exists by taking $r>0$ and $s>0$ sufficiently small, since $\operatorname{int}(\Omega)$ is open. 

Let $u^n(t)=(u_1^n,\ldots,u_N^n)(t)$ be the geodesic of $P_2(\Omega)^N$ such that $u_1^n(t)$ links $u_1^n$ to $u_2^n$, $u_2^n(t)$ links $u_2^n$ to $u_1^n$, and the other components are not moving (i.e., they are constant for $t\in[0,1]$). The components $u_k^n(t)$ are geodesics in $P_2(\Omega)$; they are balls of radius $r/2^n$ whose center is moving with constant speed. Denote by $\mathcal{E}_{12}$ the correlation energy between species 1 and 2, $\mathcal{E}_{12}(u)=\frac12\int_\Omega u_1u_2 dx$. A computation shows that
\begin{align*}
  \mathcal{E}_{12}(u^n(0)) = \mathcal{E}_{12}(u^n(1)) = 0, \quad
  \mathcal{E}_{12}(u^n(\tfrac12)) 
  = \frac{1}{\mbox{meas}(B_{r/2^n}(0))}
  \to +\infty \mbox{ as }n\to\infty.
\end{align*}
Suppose that $\mathcal{E}$ is $\lambda$-geodesically convex for some $\lambda\in\R$. Then
\begin{align*}
  \mathcal{E}_{12}(u^n(\tfrac12))
  \le \frac12\big(\mathcal{E}(u^n(0)) + \mathcal{E}(u^n(1))\big)
  - \frac{\lambda}{8}r^2 = -\frac{\lambda}{8}r^2.
\end{align*}
This leads to a contradiction, since the left-hand side diverges as $n\to\infty$. 
\end{proof}

To gain some physical intuition for the lack of semi-convexity of the energy, we consider a nonlocal approximation of the functional $\mathcal{E}_{12}$,
\begin{align*}
  \mathcal{E}_{12}^\eps(u) = \int_\Omega\int_\Omega  
  K_\eps(|x_1-x_2|)du_1(x_1)du_2(x_2),
\end{align*}
where the kernel $K_\eps:(0,\infty)\to\R$ is a convex approximation of the Dirac delta distribution. A rigorous approach for such an approximation can be found in, e.g., \cite{CEW24}. Let $(u_t)_{0\le t\le 1}$ be a geodesic linking $(u_1,u_2)$ to some $(v_1,v_2)$ with the optimal coupling $(\gamma_1,\gamma_2)$. It follows that
\begin{align*}
  \mathcal{E}_{12}^\eps(u_t) = \int_\Omega\int_\Omega
  K_\eps\big(\big|(1-t)(x_1-x_2)+t(y_1-y_2)\big|\big)
  d\gamma_1(x_1,y_1)d\gamma_2(x_2,y_2). 
\end{align*}
This expression loses its convexity if the trajectories $x_1(t)=(1-t)x_1+ty_1$ and $x_2(t)=(1-t)x_2+ty_2$ collide in the sense that $x_1(t)=x_2(t)$ for some $t\in[0,1]$. If the trajectories originate from (or terminate at) the same point and separate (or collide), the convexity is conserved. 

Collisions do not happen for the porous-medium internal energy, i.e.\ when $u_1=u_2$ and $v_1=v_2$, since the support of the optimal coupling is monotone in the sense \cite[Chap.~1]{San15}
\begin{align*}
  (y_2-y_1)\cdot(x_2-x_1)\ge 0 \quad\mbox{for all }
  (x_1,y_1),\,(x_2,y_2)\in\operatorname{supp}\gamma_1.
\end{align*}
Indeed, there exists $t\in(0,1)$ such that
\begin{align*}
  (1-t)(x_1-x_2) + t(y_1-y_2) = 0\quad\mbox{and}\quad
  (y_2-y_1)\cdot(x_2-x_1)\ge 0
\end{align*}
if and only if $x_1=x_2$ and $y_1=y_2$. Since the product metric does not account for collisions, our energy is singular. 

Instead of relying on the semi-convexity property, we adopt the flow interchange technique, introduced in \cite{MMS09} (before used in \cite{Ott98}). This approach is based on analyzing the dissipation of the energy along an auxiliary gradient flow (here associated with the entropy $\mathcal{H}$) and using the evolution variational inequality to rigorously justify the chain rule. This, in turn, leads to an entropy production inequality of the type derived in \cite[Lemma 2.10]{CaSa24}. The technique was applied, for instance, to aggregation--diffusion systems \cite{CEFF24}. 

%%%%%%%%%%%%%%%%%%%%%%%%%%%%%%%%%

\subsection{Existence analysis}\label{sec.ex}

Recall definition \eqref{3.EH} of the energy $\mathcal{E}$ and entropy $\mathcal{H}$. We prove the following theorem. 

\begin{theorem}\label{thm.ex}
Let $A=(a_{ij})$ be positive definite, $T>0$, $u^0\in L^1(\Omega)^N$ such that $\mathcal{E}(u^0)<\infty$, $\mathcal{H}(u^0)<\infty$. Then there exists a weak solution $u\in C^0([0,T];P_2(\Omega)^N)$ to \eqref{1.BT} satisfying $u\in L^\infty(0,T;L^2(\Omega)^N)\cap L^2(0,T;H^1(\Omega)^N)$ and for all $i=1,\ldots,N$ and $\phi_i\in C^\infty(\Omega)$,
\begin{align*}
  -\int_\Omega(u_i(T)-u_i^0)\phi_i dx
  + \sum_{j=1}^N\int_\Omega u_i a_{ij}\na u_j\cdot\na\phi_i dx = 0,
  \quad i=1,\ldots,N.
\end{align*}
\end{theorem}

The proof is inspired from \cite{CEFF24}. It is based on the JKO scheme. Let $0=t_0<t_1<\cdots<t_m=T$ for $m\in\N$ be a partition of time instances with the time step size $\tau_k=t_{k+1}-t_{k}>0$, and define a sequence of vector-valued probability measures $u^k\in P_2(\Omega)^N$ iteratively by
\begin{align*}
  u^{k+1} = \argmin_{u\in P_2(\Omega)^N}
  \bigg(\frac{W_2^2(u^k,u)}{2\tau_k} + \mathcal{E}(u)\bigg).
\end{align*}
This scheme is well defined since the functional $u\mapsto u^TAu$ in the energy $\mathcal{E}$ is convex in the usual Euclidean sense (as $A$ is assumed to be positive definite). We suppose that the time step size converges to zero as the partition becomes finer, i.e.
\begin{equation}\label{eq:conv_unif_timestep}
  \lim_{m \rightarrow \infty } \sup_{k=0,\ldots,m-1} \tau_k = 0.
\end{equation}
The solution to the minimization problem can be characterized in terms of the 1-convex Kantorovich potential $\phi^{k+1\to k}$ that gives the transport map from $u^{k+1}$ to $u^k$, i.e.\ $u^k = (I-\na\phi^{k+1\to k})_{\#}u^{k+1}$.

\begin{lemma}\label{lem.kant}
The Kantorovich potential $\phi^{k+1\to k}=(\phi_i^{k+1\to k})_{i=1}^N$ verifies
\begin{equation}\label{3.kant}
  \frac{\phi_i^{k+1\to k}}{\tau_k} + \sum_{j=1}^N a_{ij}u_j^{k+1}
  = C \ \mbox{on }\operatorname{supp}(u_i^{k+1}), \quad
  \frac{\phi_i^{k+1\to k}}{\tau_k} + \sum_{j=1}^N a_{ij}u_j^{k+1}
  \ge C \ \mbox{else},
\end{equation}
where the constant $C$ does not depend on $x\in\Omega$. Moreover, the map $x\mapsto \sum_{j=1}^N a_{ij}u_j^{k+1}(x)$ is Lipschitz continuous on $\operatorname{supp}(u_i^{k+1})$ and its gradient is defined in $L^2(\Omega;u_i^{k+1}dx)$.  
\end{lemma}

\begin{proof}
The proof is based on classical arguments; see, e.g., \cite[Chap.~7.2]{San15}. Therefore, we give a sketch of the proof only. Let $i\in\{1,\ldots,N\}$. We perturb the optimal element $u^{k+1}$ by defining $u_i(\eps):=(1-\eps)u_i^{k+1} + \eps\bar{u}$, where $\bar{u}\in P_2(\Omega)\cap L^\infty(\Omega)$, and $u_j(\eps):=u_j^{k+1}$ for $j\neq i$, $\eps>0$. Set $\Phi(u^k,u) = W_2^2(u^k,u)/(2\tau_k) + \mathcal{E}(u)$. By optimality,
\begin{align*}
  \frac{\Phi(u^k,u(\eps)) - \Phi(u^k,u^{k+1})}{\eps}\ge 0.
\end{align*}
Using the first variation results for the Wasserstein distance and the energy \cite[Remark 7.13, Theorem 7.17]{San15}, we can pass to the limit $\eps\to 0$ \cite[Theorem 7.20]{San15}, leading to \eqref{3.kant}. The transport $\na\phi_i^{k+1\to k}$ is bounded, since the domain is bounded; thus, the Kantorovich potential is Lipschitz continuous. Taking into account \eqref{3.kant}, the map $x\mapsto\sum_{j=1}^N a_{ij}u_j(x)$ is Lipschitz continuous too and consequently, its gradient is defined in the space $L^2(\Omega;u_i^{k+1}dx)$. 
\end{proof}

We define the piecewise constant approximation $u^{(m)}$ by
\begin{align}\label{3.um}
  u^{(m)}(t,x) = u^k(x) \quad\mbox{for }x\in\Omega,\ t\in(t_{k-1},t_k],
  \ k=1,\ldots,m.
\end{align} 
First, we prove that this sequence has a limit in $C^0([0,T];L^2(\Omega)^N)$ by applying the Arzel\`a--Ascoli theorem.

\begin{lemma}\label{lem.L2}
Let $u^0\in L^1(\Omega)^N$ satisfy $\mathcal{E}(u^0)<\infty$. Up to a subsequence, $u^{(m)}$ converges to some function $u$ in $C^0([0,T];P_2(\Omega)^N)$. The limit $u$ is H\"older continuous of index $1/2$ with values in $P_2(\Omega)^N$ (endowed with the Wasserstein metric). Moreover, there exists a constant $C>0$ depending only on $u^0$ (but not on $m$) such that 
\begin{align}\label{3.L2}
  \|u^{(m)}\|_{L^\infty(0,T;L^2(\Omega))}
  + \|u\|_{L^\infty(0,T;L^2(\Omega))} \le C.
\end{align}
\end{lemma}

\begin{proof}
Let $m\in\N$ be fixed and let $0\le i<j\le m$. The sequence $(\mathcal{E}(u^k))_k$ is nonincreasing since, by the JKO scheme,
\begin{align*}
  \mathcal{E}(u^{k+1})
  \le \mathcal{E}(u^{k+1}) + \frac{W_2^2(u^k,u^{k+1})}{2\tau_k}
  \le \mathcal{E}(u^{k}) + \frac{W_2^2(u^k,u^{k})}{2\tau_k}
  = \mathcal{E}(u^{k}).
\end{align*}
We infer that $\mathcal{E}(u^k)\le\mathcal{E}(u^0)$ for $k=1,\ldots,m$. Summing the previous inequalities over $k=i,\ldots,j-1$ leads to
\begin{align}\label{3.aux}
  \mathcal{E}(u^j) + \sum_{k=i}^{j-1}\frac{W_2^2(u^k,u^{k+1})}{2\tau_k}
  \le \mathcal{E}(u^i) \le \mathcal{E}(u^0).
\end{align}
Let $0\le s\le t\le T$ and let $k$ (respectively, $k'$) be the unique elements of $\N$ such that $s\in(t_{k-1},t_{k}]$ ($t\in(t_{k'-1},t_{k'}]$). Then, by \eqref{3.aux},
\begin{align}\label{3.Hoelder}
  W_2&(u^{(m)}(s),u^{(m)}(t)) = W_2(u^k,u^{k'})
  \le \sum_{\ell=k}^{k'-1} W_2(u^{\ell-1},u^\ell) \\
  &\le \bigg(\sum_{\ell=k}^{k'-1}\frac{W_2^2(u^k,u^{k+1})}{2\tau_k}
  \bigg)^{1/2}\bigg(\sum_{\ell=k}^{k'-1}2\tau_k\bigg)^{1/2}
  \le \sqrt{\mathcal{E}(u^0)}\sqrt{2(t-s)}. \nonumber 
\end{align} 
As $P_2(\Omega)$ is compact ($\Omega$ being compact), we can apply the Arzel\`a--Ascoli theorem to conclude the first part of the lemma. 

For estimate \eqref{3.L2}, the positive definiteness of $A$ implies for all $0\le k\le m$ that
\begin{align}\label{3.aux2}
  \|u^k\|_{L^2(\Omega)}^2 \le C\mathcal{E}(u^k) \le C\mathcal{E}(u^0).
\end{align}
Next, let $\phi\in C^0(\Omega)$ and $t\in[0,T]$. Since $(u^{(m)})$ converges in $C^0([0,T];P_2(\Omega)^N)$, we have
\begin{align*}
  \lim_{m\to\infty}\int_\Omega u^{(m)}(t)\phi dx
  = \int_\Omega u\phi dx.
\end{align*}
We deduce from \eqref{3.aux2} and the Cauchy--Schwarz inequality that
\begin{align*}
  \bigg|\int_\Omega u^{(m)}(t)\phi dx\bigg|  
  \le C\sqrt{\mathcal{E}(u^0)}\|\phi\|_{L^2(\Omega)}
\end{align*}
for any $m\in\N$. The limit $m\to\infty$ yields
\begin{align*}
  \bigg|\int_\Omega u(t)\phi dx\bigg| 
  \le C\sqrt{\mathcal{E}(u^0)}\|\phi\|_{L^2(\Omega)}.
\end{align*}
This proves $u\in L^\infty(0,T;L^2(\Omega)^N)$ and the second statement of the lemma. Moreover, \eqref{3.Hoelder} implies the H\"older continuity of index $1/2$.
\end{proof}

We also need a uniform gradient bound for $u^{(m)}$, which is derived from the flow interchange technique. In \cite[Sec.~4.1]{CEW24}, this technique is used with a suitable auxiliary flow, namely the Wasserstein gradient flow for the heat equation associated to the Boltzmann entropy. Here, we focus on the integration-by-parts property. 

\begin{lemma}\label{lem.H1}
Let $\mathcal{E}(u^0)<\infty$ and $\mathcal{H}(u^0)<\infty$. Then there exists $C>0$ only depending on $u^0$ such that
\begin{align*}
  \|\na u^{(m)}\|_{L^2(\Omega)} \le C.
\end{align*}
\end{lemma}

\begin{proof}
Let $0\le k\le m-1$ and define the geodesic linking from $u^{k+1}$ to $u^k$ by
\begin{align*}
  u(t) := (I-t\na\phi^{k+1\to k})_{\#}u^{k+1}, \quad t>0.
\end{align*}
As $\phi^{k+1 \rightarrow k}$ is $1$-convex, it is twice approximately differentiable by Alexandrov's theorem, so we can apply \cite[Lemma 10.4.4]{AGS05}:
\begin{align}\label{3.dHdt}
  \frac{d\mathcal{H}}{dt}(u(t))\bigg|_{t=0}
  = \int_\Omega u^{k+1}\cdot\Delta \phi^{k+1\to k}dx.
\end{align} 

We claim that $u^{k+1}$ is Lipschitz continuous on $\Omega$. Indeed, we know from Lemma \ref{lem.kant} that $w_i:=(Au^{k+1})_i$ is Lipschitz continuous on $\mbox{supp}(u_i^{k+1})$. Consequently, $u^{k+1} = A^{-1}w$ is Lipschitz continuous on $\bigcap_i\mbox{supp}(u_i^{k+1})$. Let $i\in\{1,\ldots,N\}$ and set $U_i:=\bigcap_{j\neq i} \mbox{supp}(u_j^{k+1})\cap\mbox{supp}(u_i^{k+1})^c$. Let $A^{-1}_{[i]}\in\R^{(n-1)\times(n-1)}$ be the matrix $A^{-1}$ with the $i$th line and column removed and let $w^{[i]}\in\R^{n-1}$ be the vector $w$ with the $i$th entry removed. The submatrix $A^{-1}_{[i]}$ is invertible since $A^{-1}$ does. Thus, $u_i^{k+1}$ is Lipschitz continuous on $U_i$. Since the union of $U_i$ equals $\Omega$, we conclude that $u^{k+1}$ is Lipschitz continuous on $\Omega$.

Therefore, we can integrate by parts in \eqref{3.dHdt} giving
\begin{align*}
  \frac{d\mathcal{H}}{dt}(u(t))\bigg|_{t=0}
  = -\int_\Omega\na u^{k+1}:\na\phi^{k+1\to k}dx,
\end{align*}
where ``:'' denotes the Frobenius matrix product (summation over both matrix indices). We use the geodesic convexity of $\mathcal{H}$ \cite[Theorem 7.28]{San15}, Lemma \ref{lem.kant}, and the positive definiteness of $A$:
\begin{align*}
  \mathcal{H}(u^k) - \mathcal{H}(u^{k+1}) 
  \ge \frac{d\mathcal{H}}{dt}(u(t))\bigg|_{t=0}
  = \tau_k\int_\Omega \na u^{k+1}:A\na u^{k+1}dx
  \ge C\tau_k\|\na u^{k+1}\|_{L^2(\Omega)}^2.
\end{align*}
It follows after summation over $k$ that
\begin{align*}
  \mathcal{H}(u^0) \ge \mathcal{H}(u^k)
   + C\sum_{\ell=1}^k\tau_{\ell-1}\|\na u^\ell\|_{L^2(\Omega)}^2,
\end{align*}
from which we infer the result.
\end{proof}

%\begin{remark}[Comparison with the entropy method]\rm
The idea of the entropy method \cite{Jue16} is to solve iteratively the system
\begin{align*}
  \frac{1}{\tau_k}(u_i^{k+1}-u_i^k) 
  = \diver\bigg(\sum_{j=1}^N u_i^{k+1}a_{ij}\na u_j^{k+1}\bigg)
  + \eps\times\mbox{higher-order term},
\end{align*}
using the Leray--Schauder fixed-point theorem. The evolution of the Boltzmann entropy $\mathcal{H}$ along this discrete dynamics (and a suitable choice of the higher-order term) yields uniform bounds in $L^2(0,T;H^1(\Omega))$, similarly to Lemma \ref{lem.H1}. After showing some bound on the discrete time derivative of $u^k$, the Aubin--Lions lemma yields the compactness in $L^2(0,T;L^2(\Omega))$. In our setting, we apply a refinement of the Aubin--Lions lemma, proved in \cite{RoSa03}, to obtain the compactness of the sequence $(u^{(m)})$. We recall this result for the convenience of the reader.

\begin{theorem}[Rossi, Savar\'e]\label{thm.RS}
Let $X$ be a separable Banach space and let be given
\begin{itemize}
\item a lower semicontinuous functional $J:X\to[0,\infty]$ with relatively compact sublevels in $X$ (i.e.\ sets of the form $J^{-1}([0,c))$ with $c>0$),
\item a pseudo-distance $g:X\times X\to[0,\infty]$, i.e., $g$ is a lower semicontinuous function such that it follows from $g(u,v)=0$ for any $u$, $v\in X$ with $J(u)<\infty$ and $J(v)<\infty$ that $u=v$.
\end{itemize}
Let $T>0$ and let $U$ be a set of measurable functions $u:(0,T)\to X$, which is tight with respect to $J$ and satisfies the weak integral equicontinuity condition, i.e.
\begin{align*}
  \sup_{u\in U}\int_0^T J(u(t))dx < \infty, \quad
  \lim_{h\to 0}\sup_{u\in U}\int_0^{T-h}g(u(t+h),u(t))dt = 0.
\end{align*}
Then $U$ contains a subsequence $(u_m)_{m\in\N}$ which converges in measure to a measurable function $u:(0,T)\to X$, i.e.
\begin{align*}
  \lim_{m\to\infty}\big|\big\{t\in(0,T):
  \|u_m(t)-u(t)\|_X\ge\delta\|\big\}\big| = 0
  \quad\mbox{for all }\delta>0.
\end{align*}
\end{theorem}

\begin{lemma}\label{lem.strong}
Up to a subsequence, the sequence $(u^{(m)})$, constructed in \eqref{3.um}, converges strongly in $L^2(0,T;L^2(\Omega)^N)$.
\end{lemma}

\begin{proof}
We apply Theorem \ref{thm.RS} with $X=L^2(\Omega)^N$, $U=\{u^{(m)}:m\in\N\}$, $J(u)=\int_\Omega |\na u|^2dx$, and $g=W_2^2$. The tightness of $U$ with respect to $J$ follows from Lemma \ref{lem.H1}, and estimate \eqref{3.aux} implies the weak integral equi-continuity condition. We conclude from Theorem \ref{thm.RS} that a subsequence of $(u^{(m)})$ (not relabeled) converges in measure to some function $u$. We claim that $(u^{(m)})$ in fact converges strongly in $L^2(\Omega)^N$ (up to a subsequence). Let
\begin{align*}
  A_\delta^m = \big\{t\in(0,T):\|u^{(m)}(t)-u(t)\|_{L^2(\Omega)}
  \ge\delta\big\}.
\end{align*}
Then $\lim_{m\to\infty}A_\delta^m=0$ for all $\delta>0$. By Lemma \ref{lem.L2}, 
\begin{align*}
  \int_0^T\int_\Omega&|u^{(m)}-u|^2 dxdt
  \le \int_{(A_\delta^m)^c}\int_\Omega|u^{(m)}-u|^2 dxdt
  + \int_{A_\delta^m}\int_\Omega|u^{(m)}-u|^2 dxdt \\
  &\le \delta^2 T + 2|A_\delta^m|\big(
  \|u^{(m)}\|_{L^\infty(0,T;L^2(\Omega)^N)}^2
  + \|u\|_{L^\infty(0,T;L^2(\Omega)^N)}^2\big)
  \le \delta^2 T + C|A_\delta^m|.
\end{align*}
Consequently, choosing $\eps>0$ and $\delta:=\eps/(2T)$, there exists $M\in\N$ such that $|A_\delta^m|\le\eps/(2C)$ for any $m\ge M$ and
\begin{align*}
  \int_0^T\int_\Omega|u^{(m)}-u|^2 dxdt \le \eps.
\end{align*}
Since $\eps>0$ is arbitrary, this shows the claim.
\end{proof}

We have sufficient compactness to finish the proof of Theorem \ref{thm.ex}. Let $i\in\{1,\ldots,N\}$ and $\phi_i\in C_c^\infty(\Omega)$. Define $v_i^{k+1}(\eps) := (I-\eps\na\phi_i)_{\#}u_i^{k+1}$ and $v_j^{k+1}(\eps) := u_j^k$ for $j\neq i$. By the definition of the JKO scheme,
\begin{align}\label{3.ineq}
  \frac{1}{2\tau_k\eps}
  \big(W_2^2(u^k,v^{k+1}(\eps))-W_2^2(u^k,u^{k+1})\big) 
  + \frac{1}{\eps}\big(\mathcal{E}(v^{k+1}(\eps))-\mathcal{E}(u^k)\big)
  \ge 0.
\end{align}
The limit of the second term can be computed explicitly \cite[Sec.~10.2]{AGS05}:
\begin{align}\label{3.aux4}
  \lim_{\eps\to 0}\frac{1}{2 \eps}
  \big(W_2^2(u^k,v^{k+1}(\eps))-W_2^2(u^k,u^{k+1})\big) 
  = -\sum_{i=1}^N\int_\Omega\na \phi_i^{k+1\to k}
  \cdot\na\phi_i u_i^{k+1}dx.
\end{align}
We know from \eqref{3.Hoelder} that the function $u^{(m)}$ is H\"older continuous in $P_2(\Omega)^N$ with index $1/2$. With the optimal coupling $\gamma=(I\times(I-\na\phi_i^{k+1\to k}))_{\#}u_i^{k+1}$, a Taylor expansion around $x$ with the integral form of the remainder leads to 
\begin{align*}
  -\int_\Omega&\phi_i(u_i^{k+1}-u_i^k)dx
  = \int_\Omega(\phi_i(y)-\phi_i(x))d\gamma(x,y) \\
  &= ´\int_\Omega\big\{\phi_i\big(x-\na\phi_i^{k+1\to k}(x)\big)
  - \phi_i(x)\}u_i^{k+1}(x)dx \\
  &= -\int_\Omega\na\phi_i(x)\cdot\na\phi_i^{k+1\to k}(x)
  u_i^{k+1}(x)dx \\
  &\phantom{xx}+ \int_\Omega\int_0^1(\na\phi_i^{k+1\to k})^T(x)
  \na^2\phi_i\big(x-s\na\phi^{k+1\to k}(x)\big)
  \na\phi_i^{k+1\to k}(x)u_i^{k+1}(x)ds dx \\
  &= -\int_\Omega\na\phi_i\cdot\na\phi_i^{k+1\to k}u_i^{k+1}dx
  + O\big(W_2^2(u_i^k,u_i^{k+1})\big) 
\end{align*}
where the last step follows from \eqref{2.W2Kant}. We infer from \eqref{3.aux4} that
\begin{align}\label{3.W1}
  \lim_{\eps\to 0}\frac{1}{2\eps}
  \big(W_2^2(u^k,v^{k+1}(\eps))-W_2^2(u^k,u^{k+1})\big) 
  = \int_\Omega&\phi_i(u_i^{k+1}-u_i^k)dx + O(W_2^2(u_i^k,u_i^{k+1})).
\end{align}

We need to compute the terms involving the energy $\mathcal{E}$. Let $i\neq j$ and compute
\begin{align*}
  \frac{1}{\eps}\int_\Omega
  \big(v_i^{k+1}(\eps)u_j^{k+1} - u_i^{k+1}u_j^{k+1}\big)dx
  &= \frac{1}{\eps}\int_\Omega(v_i^{k+1}(\eps)-u_i^{k+1})u_j^{k+1}dx \\
  &= \frac{1}{\eps}\int_\Omega\big(u_j^{k+1}(x-\eps\na\phi_i(x))
  - u_j^{k+1}(x)\big)u_i^{k+1}(x)dx.
\end{align*}
As $\phi_i$ is smooth and $u_j^{k+1}\in L^2(0,T;H^1(\Omega))$, the term 
$(u_j^{k+1}(x-\eps\na\phi_i(x)) - u_j^{k+1}(x))/\eps$ is uniformly bounded in $\eps$. Thus, it admits a subsequence that converges weakly in $L^2(\Omega)$ to $\na u_j^{k+1}\cdot\na\phi_i$ as $\eps\to 0$. This shows that
\begin{align}\label{3.W2}
  \lim_{\eps\to 0}\frac{1}{\eps}\int_\Omega
  \big(v_i^{k+1}(\eps)u_j^{k+1} - u_i^{k+1}u_j^{k+1}\big)dx
  = -\int_\Omega u_i^{k+1}\na u_j^{k+1}\cdot\na\phi_i dx.
\end{align}
The case $i=j$ can be treated in a similar way. 

We insert \eqref{3.W1} and \eqref{3.W2} (multiplied by $a_{ij}$) into \eqref{3.ineq} and sum over $j=1,\ldots,N$ and $k=0,\ldots,m-1$ to find that
\begin{align}\label{3.aux5}
  -\int_\Omega\big(u_i^{(m)}(T)-u_i^{(m)}(0)\big)\phi_i dx
  &+ \sum_{j=1}^N a_{ij}\int_0^T\int_\Omega 
  u_i^{(m)}\na u_j^{(m)}\cdot\na\phi_i dxdt \\
  &\le O\bigg(\sum_{k=0}^{m-1} \tau_k W_2^2(u_i^k,u_i^{k+1})\bigg). \nonumber 
\end{align}
We deduce from the proof of Lemma \ref{lem.L2} that
$$
  \sum_{k=0}^{m-1}\frac{W_2^2(u_i^k,u_i^{k+1})}{\tau_k} 
  \leq \mathcal{E}(u^0) < \infty,
$$
so that by condition \eqref{eq:conv_unif_timestep}:
$$
  \lim_{m \rightarrow \infty} \sum_{k=0}^{m-1}
  \tau_k W_2^2(u_i^k,u_i^{k+1}) = 0.
$$
Hence, we can pass to the limit $m\to\infty$ in inequality \eqref{3.aux5}, since (a subsequence of) $(u^{(m)})$ converges strongly in $L^2(0,T;L^2(\Omega)^N)$ by Lemma \ref{lem.strong} and $\na u^{(m)}$ converges weakly in $L^2(0,T;L^2(\Omega)^N)$ by Lemma \ref{lem.H1}. Then $u_i^{(m)}\na u_j^{(m)}$ converges weakly in $L^1(0,T;L^1(\Omega))$, showing that the second term on the left-hand side of \eqref{3.aux5} converges. For the first term, we use the convergence of $u^{(m)}$ in $C^0([0,T];P^2(\Omega)^N)$ provided by Lemma \ref{lem.L2}. This shows that
\begin{align*}
  -\int_\Omega\big(u_i(T)-u_i(0)\big)\phi_i dx
  + \sum_{j=1}^N a_{ij}\int_0^T\int_\Omega 
  u_i\na u_j\cdot\na\phi_i dxst \le 0.
\end{align*}
Replacing $\phi_i$ by $-\phi_i$ yields the reverse inequality and hence, equality holds. This ends the proof of Theorem \ref{thm.ex}.

\begin{remark}\rm
A direct use of the method of the previous proof fails if we use a general energy $\mathcal{E}$. More specifically, consider
\begin{align*}
  \mathcal{E}(u) = \int_\Omega E(u)dx,
\end{align*}
where $E:\R^n\to[0,\infty)$ is a smooth convex function. The problem originates from the difficulty of establishing the chain rule. Indeed, let $\phi_i$ be a smooth test function and define $\theta_\eps:=I-\eps\na\phi_i$,  $v_i^k(\eps)=(\theta_\eps)_{\#}u_i^k$, and $v_j^k(\eps)=u_j^k$ for $j\neq i$. Then the change-of-variable formula \eqref{2.cov} gives
\begin{align*}
  \frac{1}{\eps}\big(\mathcal{E}(v^{k+1}(\eps)) - \mathcal{E}(u^k)\big)
  &= \frac{1}{\eps}\int_\Omega\big(E(v^{k+1}(\eps))-E(u^k)\big)dx \\
  &= \frac{1}{\eps}\int_\Omega
  \bigg\{E\bigg(u_1^k(\theta_\eps(x)),\ldots,u_{i-1}^k(\theta_\eps(x)),
  \frac{u_i^k(x)}{\det\theta'_\eps(x)}, \\
  &\phantom{xxxxxxx} u_{i+1}^k(\theta_\eps(x)),
  \ldots,u_N^k(\theta_\eps(x))\bigg)\det\theta'_\eps(x)
  - E(u^k)\bigg\}dx.
\end{align*}
To pass to the limit $\eps\to 0$ on the right-hand side, we may use the monotone convergence theorem if the energy satisfies the McCann condition \cite{McC97} or a weak--strong convergence argument if the energy has a quadratic structure. For a general energy functional, passing to the limit in order to obtain a weak formulation is more delicate and remains an open problem.
\qedd\end{remark}

\begin{remark}\rm 
Geometrically speaking, the JKO scheme provides a sequence of piecewise constant trajectories associated to $t\mapsto(u^{(m)}(t),A\na u^{(m)}(t))$ in the ``tangent bundle'' $TP_2(\Omega)$. Estimate \eqref{3.aux} yields the compactness of the trajectories $u^{(m)}$, while bounds from the entropy method allow us to prove the compactness for the tangent vector trajectories $\na(Au^{(m)})$. The limiting tangent vector can be identified as $\na(Au)$, where $u$ is the limiting trajectory obtained by the Arzel\`a--Ascoli theorem. It is interesting to draw a parallel with the theory of finite-dimensional Riemannian manifolds $M$. Assume that there exists a sequence $(x_n(t))_{n\in\N}$ of absolutely continuous trajectories on $M$ (uniformly in time) and let $v_n$ be the associated vector field. If $(x_n)_{n\in\N}$ converges to $x$ in $C^0([0,T];M)$, we cannot generally conclude the compactness of $(v_n)_{n\in\N}$, since the convergence in $C^0([0,T];M)$ is not sufficiently strong. We can first identify a chart that contains all the vectors $v_n$ and $x_n$. Then we can prove the convergence on this common chart. This is precisely the approach used to pass from the curved space $P_2(\Omega)$ to the ``chart'' of distributions, i.e., we viewed the system in the flat chart of distributions with its own topology when we proved the convergence of $u_i^{(m)}\sum_{j=1}^N a_{ij}\na u_j^{(m)}$. 
\qedd\end{remark}

%%%%%%%%%%%%%%%%%%%%%%%%%%%%%%%%%

\subsection{A fourth-order Busenberg--Travis system}\label{sec.fourth}

The approach of the previous section can be extended to an energy functional containing gradients, namely
\begin{align*}
  \mathcal{E}(u) = \frac12\int_\Omega\bigg(\sum_{i,j=1}^N
  a_{ij}u_iu_j + \sum_{i=1}^N|\na u_i|^2\bigg)dx,
\end{align*}
giving the fourth-order Busenberg--Travis system
\begin{align*}
  \pa_t u_i = \diver\bigg(u_i\sum_{j=1}^n a_{ij}\na u_j
  + u_i\na\Delta u_i\bigg)\quad\mbox{on }\Omega,\ t>0,
\end{align*}
with initial and no-flux boundary conditions. This system was analyzed in \cite{CEFF24} with $a_{ij}\le 0$, giving an aggregation-diffusion model, where the fourth-order diffusion competes with backward diffusion. To be consistent with our framework, we assume that $A=(a_{ij})$ is positive definite.

The gradient term in the energy complicates the limit $m\to\infty$ performed in the previous section. Let $\phi_i$ be a smooth test function and set $\theta_\eps = I-\eps\na\phi_i$ for $i=1,\ldots,N$. Furthermore, let $v_i^k(\eps)=(\theta_\eps)_{\#}u_i^k$ and $v_j^k(\eps)=u_j^k$ for $j\neq i$. We deduce from the chain rule and \eqref{2.cov} that
\begin{align*}
  \int_\Omega&|\na v_i^k(\eps)|^2 dx
  = \int_\Omega\bigg|\na\bigg(\frac{u_i^k}{\det\theta'_\eps}
  \circ\theta_\eps^{-1}\bigg)\bigg|dx
  = \int_\Omega\bigg|(\theta_\eps^{-1})'
  \na\bigg(\frac{u_i^k}{\det\theta'_\eps}\bigg)
  \circ\theta_\eps^{-1}\bigg|^2 dx \\
  &= \int_\Omega\bigg|(\theta_\eps^{-1})'\circ\theta_\eps
  \na\bigg(\frac{u_i^k}{\det\theta'_\eps}\bigg)\bigg|^2 
  |\det\theta'_\eps|dx
  = \int_\Omega\bigg|(\theta'_\eps)^{-1}
  \na\bigg(\frac{u_i^k}{\det\theta'_\eps}\bigg)
  \bigg||\det\theta'_\eps|dx.
\end{align*}
It follows from the Taylor expansions $(\theta'_\eps)^{-1} = I + \eps\na^2\phi_i + o(\eps)$ and $\det\theta'_\eps = 1 - \eps\Delta\phi_i + o(\eps)$ after a computation that
\begin{align*}
  \int_\Omega&|\na v_i^k(\eps)|^2 dx
  = \int_\Omega\big|\na u_i^k + \eps\big(\na^2\phi_i\na u_i^k
  + \na\Delta\phi_i u_i^k + \Delta\phi_i\na u_i^k\big)
  + o(\eps)\big|^2 \\
  &\phantom{xx}\times|1-\eps\Delta\phi_i + o(\eps)|dx \\
  &= \int_\Omega|\na u_i^k|^2 dx + 2\eps\int_\Omega\big\{
  \na u_i^k\cdot\big(\na^2\phi_i\na u_i^k + \na\Delta\phi_i u_i^k
  + \Delta\phi_i\na u_i^k\big) \\
  &\phantom{xx}- \Delta\phi_i|\na u_i^k|^2\big\}dx + o(\eps) \\
  &= \int_\Omega|\na u_i^k|^2 dx + 2\eps\int_\Omega
  \na u_i^k\cdot\bigg(\na^2\phi_i\na u_i^k + \na\Delta\phi_i u_i^k
  + \frac12\Delta\phi_i\na u_i^k\bigg)dx + o(\eps).  
\end{align*}
This shows that
\begin{align}\label{3.grad}
  \frac{1}{2\eps}&\int_\Omega\big(|\na v_i^k(\eps)|^2 
  - |\na u_i^k|^2\big)dx \\
  &= \int_\Omega
  \na u_i^k\cdot\bigg(\na^2\phi_i\na u_i^k + \na\Delta\phi_i u_i^k
  + \frac12\Delta\phi_i\na u_i^k\bigg)dx + o(1). \nonumber 
\end{align}
By formal integration by parts on the right-hand side, we find that
\begin{align*}
  \int_\Omega &
  \na u_i^k\cdot\bigg(\na^2\phi_i\na u_i^k + \na\Delta\phi_i u_i^k
  + \frac12\Delta\phi_i\na u_i^k\bigg)dx \\
  &= \int_\Omega\bigg(\na u_i^k\na^2\phi_i\na u_i^k
  - \frac12\Delta\phi_i|\na u_i^k|^2 
  - u_i^k\Delta\phi_i\Delta u_i^k\bigg)dx \\
  &= \int_\Omega\bigg(\na u_i^k\na^2\phi_i\na u_i^k
  + \na\phi_i\na^2 u_i^k\na u_i^k 
  - u_i^k\Delta\phi_i\Delta u_i^k\bigg)dx \\
  &= -\int_\Omega\big(\Delta u_i^k\na u_i^k\cdot\na\phi_i
  + u_i^k\Delta\phi_i\Delta u_i^k\big)dx
  = \int_\Omega u_i^k\na\Delta u_i^k\cdot\na\phi_i dx.
\end{align*}
Consequently,
\begin{align*}
  \frac{1}{2\eps}\int_\Omega\big(|\na v_i^k(\eps)|^2 
  - |\na u_i^k|^2\big)dx
  = \int_\Omega u_i^k\na\Delta u_i^k\cdot\na\phi_i dx + o(1).
\end{align*}
To pass to the limit $m\to\infty$ in the terms involving $\na u_i^k$ in \eqref{3.grad}, we need compactness in $L^2(0,T;H^1(\Omega))$. The $L^2(0,T;L^2(\Omega))$ compactness can be shown as in the previous section, yielding $u_i^{(m)}\to u_i$ strongly in $L^2(0,T;L^2(\Omega))$, possibly for a subsequence. For the gradient term, we use the Gagliardo--Nirenberg inequality:
\begin{align*}
  \int_0^T&\|\na(u_i^{(m)}-u_i)\|_{L^2(\Omega)}^2 dt \\
  &\le C\int_0^T\|u_i^{(m)}-u_i\|_{L^2(\Omega)}
  \|\na^2(u_i^{(m)}-u_i)\|_{L^2(\Omega)}dt
  + C\int_0^T\|u_i^{(m)}-u_i\|_{L^2(\Omega)}^2 dt \\
  &\le C\|u_i^{(m)}-u_i\|_{L^2(0,T;L^2(\Omega))}\big(
  \|\na^2(u_i^{(m)}-u_i)\|_{L^2(0,T;L^2(\Omega))} 
  + \|u_i^{(m)}-u_i\|_{L^2(0,T;L^2(\Omega))}\big).
\end{align*}
Thus, to obtain strong convergence in $L^2(0,T;H^1(\Omega))$, it is sufficient to prove some bound in $L^2(0,T;H^2(\Omega))$. This is achieved by combining two ingredients. First, we infer from the flow interchange technique \cite[Sec.~4.1]{CEFF24} that there exists a constant $C>0$ independent of $m$ such that
\begin{align*}
  \int_0^T\int_\Omega|\Delta u_i^{(m)}|^2 dxdt \le C.
\end{align*}
Second, we need the following inequality:
\begin{align*}
  \int_\Omega|\na^2 u|^2 dx \le C_0\int_\Omega(\Delta u)^2 dx
  \quad\mbox{for all }u\in H^2(\Omega)
\end{align*}
and for some $C_0>0$. It holds true for the whole space $\R^d$ and the torus with equality and $C_0=1$ as well as for convex domains and functions $u\in H^2(\Omega)$ satisfying $\na u\cdot\nu=0$ on $\pa\Omega$. (The latter fact is a consequence of the inequality $\na|\na u|^2\cdot\nu\le 0$ on $\pa\Omega$; see \cite[Prop.~2.4]{LiMi13}.)

%%%%%%%%%%%%%%%%%%%%%%%%%%%%%%%%%%%%%%%%%%%%%%%%%%%%%%%%%%%%%

\section{Hyperbolic--parabolic Busenberg--Travis model}
\label{sec.hyper}

We have assumed in the previous section that the matrix $A=(a_{ij})$ is positive definite. This enabled us to derive gradient estimates for the unknowns. In this section, we consider matrices $A$ not having full rank. As an example we choose $a_{ij}=1/N$. Then \eqref{1.BT} can be written as
\begin{equation}\label{6.BT}
\begin{aligned}
  & \pa_t u_i = \diver(u_i\na p), \quad p = \frac{1}{N}\sum_{j=1}^N u_j
  \quad\mbox{in }\Omega,\ t>0, \\
  & u_i\na p\cdot\nu = 0 \quad\mbox{on }\pa\Omega,\ t>0, \quad
  u_i(0)=u^0_i\quad\mbox{in }\Omega, \ i=1,\ldots,N.
\end{aligned}
\end{equation}
We cannot apply the theory of Section \ref{sec.ex}, since the flow interchange technique gives only boundedness of $\na p$ in $L^2(0,T;L^2(\Omega))$, which is not sufficient to conclude the strong convergence of (an approximate sequence of) $u_i$. Summing equations \eqref{6.BT} over $i=1,\ldots,N$, we see that $p$ satisfies a porous-medium equation, while the equations for $u_i$ can be interpreted, for fixed $\na p$, as hyperbolic transport equations. Thus, system \eqref{6.BT} is of hyperbolic--parabolic type, and any existence analysis is nontrivial.

\subsection{A new definition of solution}
%\label{sec.new}

Because of the lack of suitable gradient estimates for $u_i$, the global existence of (possibly segregated) solutions is delicate. Including a drift term, \cite{KiMe18} proves the existence of segregated solutions in one space dimension, obtained as the limit of JKO solutions. The existence of weak solutions of bounded variation for general initial
data in one space dimension was shown in \cite{CFSS18} by using a variational splitting scheme. In the multidimensional case, global segregated Lagrangian solutions were constructed in \cite{Jac25}, and global dissipative measure-valued solutions were shown to exist in \cite{HoJu25}. Here, we present another notion of solution.

Introduce the projection operator $\pi:P_2(\Omega)^N\to P_2(\Omega)$, $\pi(u)=N^{-1}\sum_{i=1}^Nu_i$. We call the vector $u\in AC^2([0,T];P_2(\Omega)^N)$ (see Section \ref{sec.basics} for the definition of this space) a solution to \eqref{6.BT} %with no-flux boundary conditions 
if the following conditions are satisfied:
\begin{itemize}
\item The pressure $p:=\pi(u)$ is the unique solution to the porous-medium equation $\pa_t p=\diver(p\na p)$ in $\Omega$, $p\na p\cdot\nu=0$ on $\pa\Omega$, $t>0$.
\item It holds that $|u'|(t)\le \sqrt{N}|p'|(t)$ for a.e.\ $t\in[0,T]$, where $u'$ is the metric derivative.
\end{itemize}
We can interpret here $P_2(\Omega)^N$ as a fiber bundle with the base space $P_2(\Omega)$, and the projection operator equals $\pi$. The existence problem reduces to finding a time-dependent, sufficiently regular section $s$ such that $u=s(t,p(t,\cdot))$. 

\begin{theorem}
Let $u^0\in P_2(\Omega)^N$. Then there exists a solution $u\in C^0([0,T];P_2(\Omega)^N)$ to \eqref{6.BT} in the sense of the previous definition.
\end{theorem}

\begin{proof}
The existence of the pressure $p$ as the unique solution to the porous-medium equation is well known. We define
\begin{align}\label{6.min}
  u^{k+1} = \argmin_{u\in P_2(\Omega)^N,\, u\in\pi^{-1}(p(t_{k+1}))}
  W_2^2(u^k,u).
\end{align}
This equation is well defined in the sense that the minimum is attained at some point, since $P_2(\Omega)$ is compact and $\pi:P_2(\Omega)^N\to P_2(\Omega)$ is continuous, so the set $\pi^{-1}(p(t_{k+1}))$ is closed. Now, let $\phi^{k\to k+1}$ be the Kantorovich potential giving the transport from $p(t_k)$ to $p(t_{k+1})$. We define  $v_i^{k+1}:=(I-\na\phi^{k\to k+1})_{\#}u_i^k$ for $i=1,\ldots,N$. Then 
\begin{align*}
  \pi(v^{k+1}) = \frac{1}{N}\sum_{i=1}^N(I-\na\phi^{k\to k+1})_{\#}u_i^k
  = (I-\na\phi^{k\to k+1})_{\#}p(t_k) = p(t_{k+1}),
\end{align*}
and we find that (using \eqref{2.W2Kant})
\begin{align}\label{6.aux}
  W_2^2(u^k,u^{k+1}) &\le W_2^2(u^k,v^{k+1})
  \le \sum_{i=1}^N\int_\Omega|\na \phi^{k\to k+1}|^2 u_i^k dx \\
  &= N\int_\Omega|\na\phi^{k\to k+1}|^2 p(t_k) dx = NW_2^2(p(t_k),p(t_{k+1})).
  \nonumber 
\end{align}
The pressure $p$ is an element of $AC^2([0,T]; P_2(\Omega))$. Thus, inequality \eqref{6.aux} implies the equicontinuity of the sequence $(u^{(m)})$. By the Arzel\`a--Ascoli theorem, there exists a subsequence, still denoted by $(u^{(m)})$, such that $u^{(m)} \to u$ strongly in $C^0([0,T]; P_2(\Omega)^N)$ as $m \to \infty$. We deduce from the continuity of the projection operator that $p = \pi(u)$. 

To control the metric derivative, let $0 \le s < t \le T$. For a fixed $m$, let $t_j, t_k$ be the grid points such that $u^{(m)}(s) = u^j$ and $u^{(m)}(t) = u^k$. Using the triangle inequality and \eqref{6.aux}, we have
\begin{align*}
  W_2(u^{(m)}(t), u^{(m)}(s)) &\le \sum_{i=j}^{k-1} W_2(u^{i+1}, u^i) 
  \le \sqrt{N} \sum_{i=j}^{k-1} W_2(p(t_{i+1}), p(t_i)) \\
  &\le \sqrt{N} \int_{t_j}^{t_k} |p'|(\tau) d\tau.
\end{align*}
As $m \to \infty$, the grid points $t_j, t_k$ converge to $s, t$ respectively. Since $u^{(m)} \to u$ in $C^0([0,T]; P_2(\Omega)^N)$, the distance $W_2(u^{(m)}(t), u^{(m)}(s))$ converges to $W_2(u(t), u(s))$ for all $s, t \in [0,T]$. By the continuity of the integral of the $L^2(0,T)$ function $|p'|$, we obtain in the limit
\begin{align*}
  W_2(u(t), u(s)) \le \sqrt{N} \int_s^t |p'|(\tau) d\tau.
\end{align*}
Dividing by $t-s$ and taking the limit $s \to t$, we conclude that $|u'| \le \sqrt{N} |p'|$ a.e. in $[0,T]$, finishing the proof.
\end{proof}

\begin{remark}\rm 
We cannot expect the uniqueness of the solution $u$ to \eqref{6.BT} in the sense of the above definition, since the cost function in \eqref{6.min} (the Wasserstein distance) is not geodesically convex \cite[Sec.~2]{CMRS20}, i.e., $P_2(\Omega)$ is not a CAT(0) space. Moreover, the constraint $u\in\pi^{-1}(p(t_{k+1}))$ is not convex in the geodesic sense. 
\end{remark}

\begin{remark}\rm 
An alternative to scheme \eqref{6.aux} is to choose directly $u^{k+1}=v^{k+1}=(I-\na\phi^{k\to k+1})_{\#}u^k$, since this choice means that the species are transported by the pressure differences, which should better reflect the underlying physics. However, a priori both constructions are different and the answer whether they coincide lies outside of the scope of this paper.
\end{remark}

%%%%%%%%%%%%%%%%%%%%%%%%%%

\subsection{One-dimensional equations}

In the following, we discuss the one-dimensional setting, following \cite{CFSS18}. Let $\Omega=[0,1]$, $N=2$, and consider initial data in $L^\infty(\Omega)\cap BV(\Omega)$, where $BV(\Omega)$ is the space of functions with bounded variation. We describe briefly the method and explain why one cannot easily generalize it to arbitrary dimensions. The method relies on the following change of unknowns (already used in \cite{BGHP85}),
\begin{align*}
  p = u_1+u_2, \quad r = \frac{u_1}{p},
\end{align*}
which leads to the decoupled transport--diffusion system
\begin{align*}
  \pa_t p = \pa_x(p\pa_x p), \quad \pa_t r = \pa_x p\pa_x r
  \quad\mbox{in }\Omega,\ t>0.
\end{align*}
Intuitively, the BV norm of the solution to the porous-medium equation decreases over time. To prove it rigorously, the authors of \cite{CFSS18} use the JKO scheme as in Section \ref{sec.ex}. We recall the argument in a few formal steps. The first step is to prove that $\pa_x\phi_1^{k+1\to k}=\pa_x\phi_2^{k+1\to k}=-\pa_x p^{k+1}$. We set $\pa_x\phi^{k+1\to k}:=\pa_x\phi_i^{k+1\to k}$ for $i=1,2$. In the second step, we use the five gradient inequality \cite[Lemma 3.2]{DMSV16}, which reads as
\begin{align*}
  \int_\Omega\big(\pa_x u_i^{k+1}h'(\pa_x\phi^{k+1\to k})
  + \pa_x u_i^{k} h'(\pa_x\phi^{k\to k+1})\big)dx \ge 0
\end{align*}
for all increasing, convex, radially symmetric functions $h:\R\to\R$. We choose $h(x)=|x|$ and sum over $i=1,2$:
\begin{align*}
  0 &\le\int_\Omega\bigg(\pa_x p^{k+1}
  \frac{\pa_x\phi^{k+1\to k}}{|\pa_x\phi^{k+1\to k}|}
  + \pa_x p^{k}
  \frac{\pa_x\phi^{k\to k+1}}{|\pa_x\phi^{k\to k+1}|}\bigg)dx \\
  &=\int_\Omega\bigg(-\pa_x p^{k+1}\frac{\pa_x p^{k+1}}{|\pa_x p^{k+1}|}
  + \pa_x p^k\frac{\pa_x p^k}{|\pa_x p^k|}\bigg)dx.
\end{align*}
This implies that
\begin{align*}
  \int_\Omega |\pa_x p^{k+1}|dx \le \int_\Omega|\pa_x p^k|dx, 
\end{align*}
and we infer the BV estimate. To estimate the variable $r^{k+1}$, we remark that
\begin{align}\label{6.r}
  r^{k+1} = \frac{u_1^{k+1}}{p^{k+1}}
  = \frac{(I-\pa_x\phi^{k\to k+1})_{\#}u_1^k}{
  (I-\pa_x\phi^{k\to k+1})_{\#}p^k}
  = \frac{u_1^k\circ(I-\pa_x\phi^{k\to k+1})}{
  p^k\circ(I-\pa_x\phi^{k\to k+1})}
  = r^k\circ(I-\pa_x\phi^{k\to k+1}).
\end{align}
Recall that the total variation of a function $g$ is defined by
\begin{align}\label{6.TV}
  TV(g) = \sup_{P_n}\sum_{i=0}^n|g(x_{i+1})-g(x_i)|,
\end{align}
where $P_n$ is the set of partitions $0=x_0<x_1<\cdots<x_n=1$. As the transport map $I-\pa_x\phi^{k\to k+1}$ is increasing \cite[Sec.~2.1]{San15}, we conclude from \eqref{6.r} that $TV(r^{k+1})\le TV(r^k)$, giving a BV bound for $r^{k+1}$. Consequently, $u_1^{k+1}=r^{k+1}\cdot p^{k+1}$ and $u_2^{k+1} = p^{k+1} - u_1^{k+1}$ are of bounded variation as $(BV,\cdot)$ has a ring structure. 

These bounds are sufficient to conclude the relative compactness of the sequence $(u^{(m)})$. Indeed, using the Arzel\`a--Ascoli theorem as in the proof of Lemma \ref{lem.L2}, there exists a subsequence that converges in $C^0([0,T];P_2(\Omega))$ to some function $u$. The pointwise in time convergence in $L^1(\Omega)$ follows from the boundedness in $BV(\Omega)\cap L^\infty(\Omega)$. The $L^\infty(\Omega)$ bound implies the pointwise in time convergence in $L^2(\Omega)$. Finally, by dominated convergence, we conclude the strong convergence in $L^2(0,T;L^2(\Omega))$. For details, we refer to \cite{CFSS18}. 

\begin{remark}\rm 
The one-dimensional characterization of $BV(\Omega)$ given by \eqref{6.TV} prevents the generalization of the previous proof to several space dimensions. In the $N$-species case, the change of unknowns becomes 
\begin{align*}
  p := \sum_{i=1}^n u_i, \quad r_i := \frac{u_i}{p} \quad
  \mbox{for }i=1,\ldots,N-1.
\end{align*}
For given $(p,r_1,\ldots,r_{N-1})$, the inverse transformation is given by $u_i = pr_i$ for $i=1,\ldots,N-1$ and $u_N = p - \sum_{i=1}^{N-1} u_i$. This can be generalized even to systems of the form $\pa_t u_i = \diver(k_iu_i\na p(u)) = 0$ with $p(u) = \sum_{i=1}^N u_i$ and $k_i>0$; see \cite[Theorem 2.2]{DHJ23}. 
\qedd\end{remark}

%%%%%%%%%%%%%%%%%%%%%%%%%%%%%%%%%

\section{Complementary effects from the SKT model}\label{sec.SKT}

The Shigesada--Kawasaki--Teramoto (SKT) equations are defined by
\begin{align}\label{7.SKT}
  \pa_t u_i = \Delta(u_i p_i(u)) 
  = \diver(u_i\na p_i(u) + p_i(u)\na u_i)\quad\mbox{in }\Omega,
  \ t>0,\ i=1,\ldots,N,
\end{align}
with initial and no-flux boundary conditions, where $p_i(u)=\sum_{i=1}^N a_{ij}u_j$. The model has been first suggested in \cite{SKT79}. The existence of global weak solutions (under a certain detailed balance condition on the coefficients $a_{ij}$) was proved in \cite{CDJ18}. Compared to the Busenberg--Travis equations \eqref{1.BT}, model \eqref{7.SKT} contains the additional diffusion term $\diver(p_i(u)\na u_i)$, which is expected to decrease the segregation effect. 

The SKT model can be derived from stochastic interacting many-particle systems in the diffusion limit. The approach of \cite{CDHJ21} considers diffusion terms in the stochastic model that depend nonlinearly on the interactions between the individuals, while the work \cite{DDD19} proves the mean-field limit of a microscopic many-particle Markov process on a discrete one-dimensional space, leading to a finite-difference discretization of \eqref{7.SKT}. 

In order to understand this additional diffusion effect given by $\diver(p_i(u)\na u_i)$, we give another approach: The state of the particle system is given by a probability measure $p(x_1,\ldots,x_N)$ on the set of positions $x_i$ of particles of type $i$, meaning that we have $N$ stochastic particles of different type. For the sake of simplicity, we restrict ourselves to the case of two species, $N=2$. Adapting the work \cite{DDD19} to our context, the equation satisfied by $p$ reads as
\begin{align}\label{7.p}
  \pa_t p = \diver(M(|x_1-x_2|)p\na p) \quad\mbox{in }\Omega,\ t>0,
\end{align}
supplemented by no-flux boundary conditions, where the mobility function $M:[0,\infty)\to(0,\infty)$ is an approximation of the Dirac delta distribution at $|x_1-x_2|=0$. 

If the initial data $u^0=(u_1^0,u_2^0)$ is chaotic in the sense that $p(0,x_1,x_2) = (u_1\otimes u_2)(x_1,x_2) = u_1^0(x_1)\cdot u_2^0(x_2)$ is expressed as an elementary tensor product, we cannot expect that this structure is preserved in time. We confirm this statement by a numerical simulation. Let the initial data be given by a Gaussian function with diagonal covariance centered at $(-2,2)$ and with a mobility $M$ that is a Gaussian approximation of the delta distribution at $x_1=x_2$. Equation \eqref{7.p} is discretized by an explicit-in-time centered finite-difference scheme. Figure \ref{fig:marginales} shows the marginals $u_1$ and $u_2$ and the joint density $p$ at time $t=1$. We observe that the particles become correlated during the dynamics. As soon as $p$ has support on the set where $x_1=x_2$, the relative entropy
\begin{align*}
  H(p\,||\,u_1\otimes u_2) = \int_\Omega p\log
  \bigg(\frac{p}{u_1\otimes u_2}\bigg)dx
\end{align*}
increases in time; see Figure \ref{fig:entropie}. 

\begin{figure}[htb]
    \centering
    \includegraphics[width=0.95\textwidth]{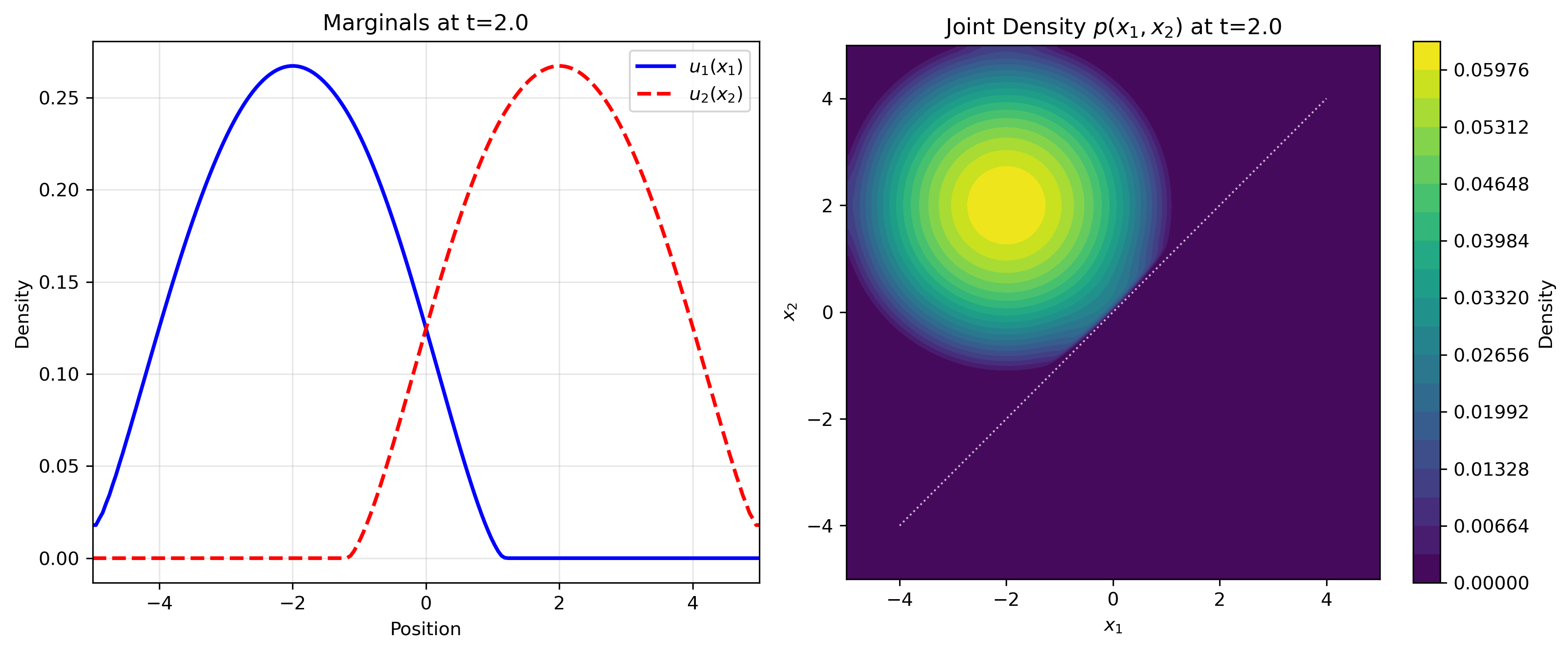}
    \caption{Marginals $u_1(x_1)$, $u_2(x_2)$ (left) and joint density $p(x_1,x_2)$ (right) at time $t=1$.}
    \label{fig:marginales}
\end{figure}
\begin{figure}[htb]
    \centering
    \includegraphics[width=0.7\textwidth]{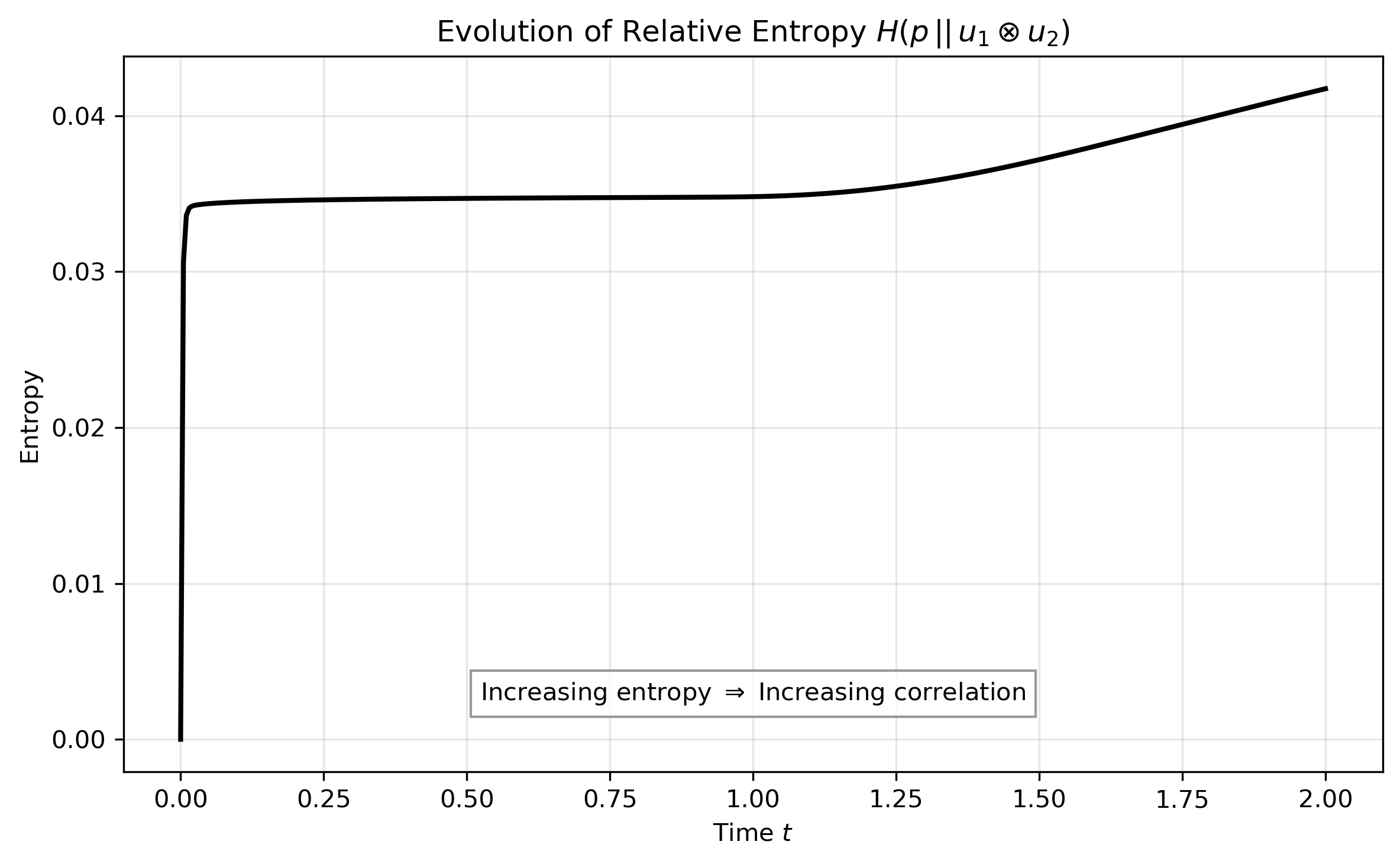}
    \caption{Relative entropy $H(p \,||\, u_1 \otimes u_2)$.}
    \label{fig:entropie}
\end{figure}

It is possible to enforce decorrelation by adding a stiff collision term ensuring the constraint $p=u_1\otimes u_2$. To this end, we endow the space $P_2(\Omega)^2 = P_2(\Omega)\times P_2(\Omega)$ with the metric on $P_2(\Omega^2)=P_2(\Omega\times\Omega)$ such that
\begin{align*}
  \widetilde{W}_2^2((u_1,u_2),(v_1,v_2))
  = \inf_{\gamma\in\Pi(u_1\otimes u_2,v_1\otimes v_2)}
  \int_{\Omega^4}d_{\Omega^2}^2((x_1,x_2),(y_1,y_2))
  d\gamma(x_1,x_2,y_1,y_2),
\end{align*}
where $d_{\Omega^2}$ is the geodesic distance on the Riemannian manifold $(\Omega^2,M^{-1})$.

\begin{proposition}\label{prop.cl}
Assume that the mobility $M$ is a smooth positive function such that $M\le C_M$ in $[0,\infty)$ for some $C_M>0$. Then the space $P_2(\Omega)^2$ is closed in $(P_2(\Omega^2),\widetilde{W}_2)$. 
\end{proposition}

\begin{proof}
Define the projections $\pi_i:(\Omega^2,d_{\Omega^2})\to (\Omega,d_\Omega)$ by $\pi_1(x,y)=x$ and $\pi_2(x,y)=y$, where $d_\Omega$ is the Euclidean distance. These functions are Lipschitz continuous. Indeed, we have $d(\pi_1)(x,y)(dx,dy)= dx$ and
\begin{align*}
  |d(\pi_1)(x,y)(dx,dy)| = |dx|
  \le \sqrt{C_M}\sqrt{M(x,y)^{-1}(dx^2+dy^2)},
\end{align*}
and similarly for $\pi_2$. This yields for all $x_i$, $y_i\in\Omega$ and $i=1,2$ that
\begin{align*}
  |x_i-y_i| \le \sqrt{C_M}d_{\Omega^2}((x_1,x_2),(y_1,y_2))
\end{align*}
and consequently, by definition of the Wasserstein distance, for all $\mu$, $\nu\in P_2(\Omega)^2$,
\begin{align*}
  W_2((\pi_i)_{\#}\mu,(\pi_i)_{\#}\nu) 
  \le \sqrt{C_M}\widetilde{W}_2(\mu,\nu),\quad i=1,2.
\end{align*}

Now, let $(p_n)_{n\in\N}$ be a converging sequence in $P_2(\Omega)^2$, i.e.\ $p_n = u_{1,n}\otimes u_{2,n}$, where $u_{i,n}=\pi_i\# p_n:=(\pi_i)_{\#}p_n$ for $i=1,2$ are the marginals of $p_n$. Denoting by $p$ the corresponding limit, we need to show that $p\in P_2(\Omega)^2$. It follows from the previous inequality that
\begin{align*}
  W_2(u_{i,n},\pi_i\# p) \le \sqrt{C_M}\widetilde{W}_2(p_n,p)
  \to 0 \quad\mbox{as }n\to\infty,\ i=1,2,
\end{align*}
which shows that $\lim_{n\to\infty}u_{i,n}=\pi_i\# p$. Let $\phi$, $\psi\in C^2(\Omega)$. Then the limit in
\begin{align*}
  \int_{\Omega^2}\phi(x)\psi(y)dp_n(x,y) 
  = \int_{\Omega^2}\phi(x)\psi(y)d(u_{1,n}\otimes u_{2,n})(x,y)
  = \int_\Omega\phi du_{1,n}\int_\Omega\psi du_{2,n}
\end{align*}
leads to
\begin{align*}
  \int_{\Omega^2}\phi(x)\psi(y)dp(x,y) 
  = \int_\Omega\phi d(\pi_1\#p)\int_\Omega\psi d(\pi_2\#p).
\end{align*}
This shows that $p\in P_2(\Omega)^2$ and finishes the proof.
\end{proof}

We define a limiting curve as the successive proximal gradient (or obtained from the JKO scheme) of the energy $\mathcal{E}(p)=\frac12\int_{\Omega^2}p^2 d(x,y)$ with respect to the metric given by $\widetilde{W}_2$. The equations satisfied by the marginals read as
\begin{equation}\label{7.eq}
\begin{aligned}
  \pa_t u_1 &= \diver\bigg(u_1\bigg(\int_\Omega M(|x-y|)
  u_2^2 dy\bigg)\na u_1\bigg), \\
  \pa_t u_2 &= \diver\bigg(u_2\bigg(\int_\Omega M(|x-y|)
  u_1^2 dy\bigg)\na u_2\bigg) \quad\mbox{in }\Omega,\ t>0.
\end{aligned}
\end{equation}
This is a system of quadratic porous-medium equations that are nonlocally coupled by the mobility. We cannot prove the existence of a curve of maximal slope $p$ in the same way as in \cite{AGS05}, since $(P_2(\Omega)^2,\widetilde{W}_2)$ is not a geodesic space. Indeed, the set of tensor products of measures in $P_2(\Omega^2)$ 
is generally not geodesically convex unless the underlying distance 
$d_{\Omega^2}$ is a product distance.

\begin{theorem}
Let $c_M\le M \leq C_M$ in $[0,\infty)$ for some positive constant $c_M$. Then there exists a weak solution to \eqref{7.eq} in $C^0([0,T];(P_2(\Omega)^2,\widetilde{W}_2))$. 
\end{theorem}

\begin{proof}
With the notation of Section \ref{sec.ex}, we define the JKO scheme on a temporal partition by
\begin{align*}
  p^{k+1} := \argmin_{p\in P_2(\Omega)^2}
  \bigg(\frac{\widetilde{W}_2(p^k,p)}{2\tau_k} + \mathcal{E}(p)\bigg).
\end{align*}
Using the flow interchange technique with the usual Boltzmann entropy and the fact that $M \geq c_M >0$, we can prove the uniform boundedness of the sequence of piecewise-in-time constant functions $p^{(m)}$ in $L^2(0,T;H^1(\Omega^2))\cap L^\infty(0,T;L^2(\Omega^2))$. Let $i\in\{1,2\}$, $\phi_i\in C_c^\infty(\Omega)$ and define $v_i^{k+1}(\eps) := (\pi_i \exp(-\varepsilon M \nabla \phi_i))_{\#}p^{k+1}$ and $v_j^{k+1}(\eps):=u_j^{k+1}$ for $j\neq i$, $p^{k+1}(\eps):=v_1^{k+1}(\eps)\otimes v_2^{k+1}(\eps)$. Here, $\pi_i : \Omega^2 \rightarrow \Omega$ denotes the canonical projection onto the $i$th component. 

By the definition of the JKO scheme, we have
\begin{align}\label{eq:limit_JKO_Wtilde}
  \frac{1}{2\tau_k\eps}\big(\widetilde{W}_2^2(p^k,p^{k+1}(\eps))
  - \widetilde{W}_2^2(p^k,p^{k+1})\big) 
  + \frac{1}{\eps}\big(\mathcal{E}(p^{k+1}(\eps))
  -\mathcal{E}(p^{k+1})\big)\ge 0.
\end{align}
For the ease of reading, we consider the case $i=1$ and $j=2$ in the following computations. The limit of the first term as $\eps\to 0$ is more delicate to establish than in the proof of Theorem \ref{thm.ex}, since we are working in a Riemannian manifold rather than in the Euclidean setting. The idea is to use the Riemannian version of the differentiability of the Wasserstein distance given in \cite[Theorem 23.9]{villani2009optimal}. Denoting by $\phi^{k+1 \rightarrow k}$ the Kantorovich potential from $p^{k+1}$ to $p^k$ and by $g$ the metric associated to $(\Omega, M^{-1})$ (with gradient $\nabla_g = M \nabla$), we find that
\begin{align*}
  \lim_{\varepsilon \rightarrow 0 } \frac{1}{2\eps}\big(\widetilde{W}_2^2(p^k,p^{k+1}(\eps))
  - \widetilde{W}_2^2(p^k,p^{k+1})\big) 
  &= - \int_{\Omega}\int_\Omega \big\langle \nabla_{g} \phi_1(x_1), \nabla_{g} \phi^{k+1 \rightarrow k} \big\rangle_{g} p^{k+1}dx_1 dx_2\\
  &= - \int_{\Omega}\int_\Omega  M \nabla \phi_1(x_1)\cdot \nabla_{x_1} \phi^{k+1 \rightarrow k} p^{k+1} dx_1 dx_2.
\end{align*}
Aside from this, and by a Taylor expansion similar to the one derived in the proof of Theorem \ref{thm.ex}, it holds for $\gamma \in \Gamma_0(p^k, p^{k+1})$ (the set of optimal transport plans coupling $p^k$ and $p^{k+1}$) that
\begin{align*}
  -\int_{\Omega} &\int_\Omega\phi_1(x_1) (p^{k+1}-p^k)(x_1,x_2)dx_1dx_2
  = \int_\Omega(\phi_1(y_1)-\phi_1(x_1))d\gamma(x_1,x_2, y_1,y_2) \\
  &= ´\int_{\Omega}\int_\Omega\big\{\phi_1\big( \pi_1 \exp_{x_1,x_2}(- \nabla_g \phi^{k+1 \rightarrow k}) \big)
  - \phi_1(x_1)\}p^{k+1}(x_1,x_2)dx_1dx_2 \\
  &= -\int_{\Omega}\int_\Omega M \na\phi_1(x_1) \cdot\na\phi^{k+1\to k}(x_1,x_2) p^{k+1}(x_1,x_2)dx_1dx_2
  + O(W_2^2(p^k,p^{k+1})).
\end{align*}
Consequently, 
\begin{align*}
  \lim_{\varepsilon \rightarrow 0 } 
  \frac{1}{2\eps}\big(\widetilde{W}_2^2(p^k,p^{k+1}(\eps))
  - \widetilde{W}_2^2(p^k,p^{k+1})\big) 
  &=  -\int_{\Omega}\int_\Omega\phi_1(x_1) (p^{k+1}-p^k)dx_1dx_2 \\
  &\phantom{xx}+ O(W_2^2(p^k,p^{k+1})).
\end{align*}
For the energy term in \eqref{eq:limit_JKO_Wtilde}, the derivation relies on a similar argument as in the proof of Theorem \ref{thm.ex}. Hence, the limit $\eps\to 0$ in \eqref{eq:limit_JKO_Wtilde} yields
\begin{align*}
  -\int_{\Omega}&\int_\Omega\phi_1(x_1)(p^{k+1}-p^k)(x_1,x_2)dx_1 dx_2 \\
  &\phantom{xx}+ \tau_k\int_{\Omega}\int_\Omega M(|x_1-x_2|)
  \na p^{k+1}(x_1,x_2)\cdot\na\phi_1(x_1)p^{k+1}(x_1,x_2)dx_1 dx_2 \\
  &\ge O\bigg(\sum_{k=0}^{m-1} \tau_k W_2^2(p^k,p^{k+1})\bigg).
\end{align*}
We sum over $k=0,\ldots,m-1$, then use the compactness of $(p^{(m)})$, given by the bounds in $L^2(0,T;H^1(\Omega^2))\cap L^\infty(0,T;L^2(\Omega^2))$, and finally the compactness in $C^0([0,T];P_2(\Omega)\times P_2(\Omega))$ to infer that in the limit $m\to\infty$,
\begin{align*}
  -\int_{\Omega}&\int_\Omega\phi_1(p(T)-p(0))(x,y)dxdy \\
  &+ \int_0^T\int_{\Omega}\int_\Omega 
  M(|x-y|)\na p(x,y)\cdot\na\phi_1(x)p(x,y)dx dydt \ge 0,
\end{align*}
where we used the same trick as the proof in Theorem \ref{thm.ex} to eliminate the remainder term. By changing $\phi_1$ to $-\phi_1$, we obtain equality. An analogous equality holds for $\phi_2$. 

We deduce from the convergence in $C^0([0,T];P_2(\Omega)^2)$ and Proposition \ref{prop.cl} that $p=u_1\otimes u_2$, where $u_1$ and $u_2$ are the marginals of $p$. This gives
\begin{align*}
  -\int_{\Omega}&\phi_1(x)(u_1(x,T)-u_1(x,0))dx \\
  &+ \int_0^T\int_{\Omega^2}M(|x-y|)u_1(x)u_2(y)^2\na u_1(x)\cdot
  \na\phi_1(x)d(x,y)dt = 0, \\
  -\int_{\Omega}&\phi_2(y)(u_2(y,T)-u_2(y,0))dy \\
  &+ \int_0^T\int_{\Omega^2}M(|x-y|)u_2(y)u_1(x)^2\na u_2(y)\cdot
  \na\phi_2(y)d(x,y)dt = 0,
\end{align*}
which corresponds to the weak formulation of \eqref{7.eq}.
\end{proof}

\begin{remark}\rm
An alternative approach to prove the existence of solutions to \eqref{7.eq} could be a fixed-point argument, defining the fixed-point operator $S:P_2(\Omega)^2\to P_2(\Omega)^2$, $(u_1,u_2)\mapsto(v_1,v_2)$, where $(v_1,v_2)$ is the unique solution to
\begin{align*}
  \pa_t v_1 &= \diver\bigg(v_1\bigg(\int_\Omega 
  M(|x-y|)u_2^2 dy\bigg)\na v_1\bigg), \\
  \pa_t v_2 &= \diver\bigg(v_2\bigg(\int_\Omega 
  M(|x-y|)u_1^2 dx\bigg)\na v_2\bigg)\quad\mbox{in }\Omega,
  \ t>0.
\end{align*}
To prove that $S$ is well defined, one may rely on the results of \cite{Kop17}, which studies gradient flows on time-varying metric spaces. More specifically, applying these results requires establishing the logarithmic time differentiability of the underlying metric on $\Omega$. This means that we have to prove the bound
\begin{align*}
  \bigg|\frac{d}{dt}\int_\Omega M(|x-y|)u_1(t,x)^2 dx\bigg|
  \le C\int_\Omega M(|x-y|) u_1(t,x)^2 dx
\end{align*}
(and similarly for $u_2$), which requires estimates on $\pa_t u_i$ that are not obvious to derive.
\end{remark}

\begin{remark}\rm
When we choose the energy $\mathcal{E}(p) = \int_{\Omega^2}p\log pd(x,y)$, we may adapt the arguments developed above. This yields the existence of a solution to the system
\begin{align*}
  \pa_t u_1 &= \diver\bigg(\bigg(\int_\Omega M(|x-y|)u_2 dy\bigg)
  \na u_1\bigg), \\
  \pa_t u_2 &= \diver\bigg(\bigg(\int_\Omega M(|x-y|)u_1 dy\bigg)
  \na u_2\bigg) \quad\mbox{in }\Omega,\ t>0.
\end{align*}
If $M$ is close to the delta distribution, this system can be seen as an approximation of
\begin{align*}
  \pa_t u_1 = \diver(u_2\na u_1), \quad
  \pa_t u_2 = \diver(u_1\na u_2).
\end{align*}
\end{remark}

%%%%%%%%%%%%%%%%%%%%%%%%%%%%%%%%%

\end{document}